\documentclass[10pt,a4paper]{article}
\usepackage[utf8]{inputenc}
\usepackage{lineno}
\usepackage{amsmath}
\usepackage{mathtools}
\usepackage{amsthm}
\usepackage{setspace}

\newtheorem{definition}{Definition}[section]
\newtheorem*{remark}{Remark}
\newtheorem{lemma}{Lemma}[section] 
\newtheorem{corollary}{Corollary}[section] 
 
\newtheorem{proposition}{Proposition}[section]

\usepackage{amsfonts}
\usepackage{float}
\usepackage{amssymb}
\usepackage{makeidx}
\usepackage{graphicx}
\usepackage{lmodern}
\usepackage{hyperref}
\usepackage[left=1in,right=1in,top=1in,bottom=1in]{geometry}
\title{On Explicit Stochastic Differential Algebraic Equations\footnote{This work was partially supported by the grant SERB-EMR/2016/002022-EEC of Science and Engineering Research Board, Government of India.}}
\author{Sumit Suthar\footnote{Email: sutharsumit@iisc.ac.in, sumitsuthar@live.in}, Soumyendu Raha\footnote{Email: raha@iisc.ac.in}\\ {\small Indian Institute of Science, Bengaluru}}
\date{14 June 2023}
\usepackage{algorithmic}
\begin{document}
\maketitle
\begin{abstract}
Dynamical systems that are subject to continuous uncertain fluctuations can be modelled using Stochastic Differential Equations (SDEs). Controlling such system results in solving path constrained SDEs. Broadly, these problems fall under the category of Stochastic Differential-Algebraic Equations (SDAEs). In this article, the focus is on combining ideas from the local theory of Differential-Algebraic Equations with that of Stochastic Differential Equations. The question of existence and uniqueness of the solution for SDAEs is addressed by using contraction mapping theorem in an appropriate Banach space to arrive at a sufficient condition. From the geometric point of view, a necessary condition is derived for the existence of the solution. It is observed that there exists a class of completely high index SDAEs for which there is no solution. Hence, techniques to find approximate solution of completely high index equations are presented. The techniques are illustrated with examples and numerical computations.\\
\textbf{Keywords:} Stochastic Differential Algebraic Equations (SDAE), completely high index Stochastic Differential Algebraic Equations, constrained Stochastic Differential Equations, constrained dynamics, dynamical systems.
\end{abstract}

\section{Introduction}
Dynamical systems are usually modelled using differential equations. When operational conditions are imposed on the system (e.g. safety or performance related constraints), it often takes the form of algebraic equation. Such constrained equation together with the original differential equation is called Differential - Algebraic Equation (DAE). From the application point of view, DAE is often used in control of physical/engineering systems such as in robotics \cite{campbell1995solvability,krishnan1994tracking,kumar1995feedback}, power engineering \cite{schiffer2016stability, huang2020optimal}, aerodynamics \cite{uppal2017trajectory}, satellite applications \cite{fox2000numerical} etc. Other than control, DAE also finds application in areas such as biochemistry \cite{strychalski2010simulating,wiechert2009modeling}, ecological sciences \cite{jiang2016modeling}, astronomy and astrophysics \cite{ashby1996computing}, electromagnetism \cite{garcia2018systems}etc.

In most of the applications, the system usually suffers from continuous uncertain fluctuations. Systems with such uncertain fluctuations can be modelled using Stochastic Differential Equations (SDEs). However, operational constraints cannot be neglected if they exists. This results in finding solution of a constrained SDE. The Stochastic Differential Equation along with the algebraic constraint equation, forms a system of Stochastic Differential-Algebraic Equations.

There is abundant literature on both theoretical and numerical aspects of DAEs. \cite{ascher1998computer, brenan1995numerical, petzold1982differential} are some references for the numerical aspects of DAEs. DAEs of completely high index or singular (implicit) differential equations have also been studied as analysis of differential equations on manifolds, e.g. in \cite{reich1990geometrical,rheinboldt1984differential}. Like DAEs, the application of SDAEs is not limited to control but, can be found in many other areas of Science and Engineering. We find that despite having enormous potential in application, unlike DAEs, SDAEs do not have a well formed theory and attempts at developing the theory are limited to SDAEs of index 1 with noiseless constraints.
In \cite{storm1995stochastic} existence and uniqueness is established for SDEs that are subjected to convex constraint. \cite{1967stochastic} is a good reference that uses ideas from stochastic stability to solve SDEs constrained to a set. However, our interest in SDAEs implies that we are looking at a special constraint set that is given in form points that satisfy the constraint algebraic equations. An attempt towards developing a stochastic analogue for the theory of Differential-Algebraic Equations (DAEs) can be found in \cite{winkler2003stochastic, winkler2006stochastic, cong2010stochastic}. However, the work is limited to noiseless constraints and SDAEs of index 1. \cite{zhang2014stability,cong2012lyapunov} discuss stability analysis of the trivial index 1 system perturbed by noise. There have also been attempts toward numerical simulation of index 1 systems in \cite{sickenberger2006efficient, penski1999analysis, winkler2006stochastic}, which essentially reduces to simulation of stiff SDEs. 

Despite having enormous potential in application, unlike DAEs, SDAEs do not have a well formed theory and attempts at developing the theory are limited to SDAEs of index 1. In this article we extend the ideas of differential-algebraic equations to the stochastic setting. We also give computational methods for approximately solving SDAEs. The main results obtained in this article includes:
\medskip
\begin{itemize}
\item Notion of Index for SDAEs.
\item Sufficient condition for existence and uniqueness of solution of SDAE, and conditions for existence of ill-posed equations for which there is no solution in Ito SDE form.
\item Computational methods to find exact and approximate solution for completely high index equations.
\end{itemize}
\medskip
Basic terminologies and notations used in this article can be found in section \ref{section:basic_def}. As SDAE are fundamentally different from deterministic DAEs, before developing computational techniques it is necessary to establish the class of equations that may or may not have a solution. These conditions can be found as necessary and sufficient condition in section \ref{sec:n&Scondition}. Existence of \textit{ill-posed} equations, which do not have a solution, is established in section \ref{sec:necessary condition}. In section \ref{sec:preliminary results} an attempt is made to find exact solution for completely high index equations by reducing it to SDE. Based on this, we also give the notion of index of SDAEs. It is observed that not all completely high index equations can be solved using this technique. Moreover, existence of ill-posed equations pose a challenge in developing a general computational technique. Hence, the article also focuses on finding techniques to obtain an approximate solution. Section \ref{sec:gen_ito} and section \ref{sec:bounded_ito} presents methods for numerical computation. The computational methods are illustrated with examples. The article ends with comparison between the techniques for finding approximate solution (section \ref{sec:comparison of techni.}) and with some concluding remarks in section \ref{sec:conclusions}.

\section{Preliminaries}
\label{section:basic_def}
We define a multidimensional SDAE  as
\begin{subequations}
\label{eq:SDAE1}
\begin{equation}
\label{eq:SDAE-SDE}
dx = f(x,u)dt + \sigma(x,u)dW_t,  
\end{equation}
\begin{equation}
\label{eq:SDAE-AE}
g(x,u) = - \int_0^t\Gamma(x(s),u(s))dW_s;
\end{equation}
\end{subequations}
where $W_t$ is a d-dimensional Wiener process and $\{\mathcal{F}_t\}$ is filteration generated by $W_t$, $f:U\times V\to \mathbb{R}^n$, and $\sigma:U\times V\to L(\mathbb{R}^d,\mathbb{R}^n)$. $x(t,\omega) \in \mathbb{R}^n$ is called \textit{{\textbf{state variable}}} and $u(t,\omega) \in \mathbb{R}^m$ is called \textit{{\textbf{algebraic variable}}}. The second equation is the constraint equation (also called the \textit{\textbf{algebraic equation}}) with $g:U\times V \to \mathbb{R}^p$ and $\Gamma:U\times V \to L(\mathbb{R}^d,\mathbb{R}^p)$. All the stochastic integrals are Ito integrals. $U$ and $V$ are connected neighbourhoods of the initial condition $x(0)$ and $u(0)$ respectively. Equation \eqref{eq:SDAE-SDE} and equation \eqref{eq:SDAE-AE} are together called \textit{\textbf{explicit Stochastic Differential-Algebraic Equations}}, as opposed to implicit SDAEs that are usually represented as singular SDEs discussed in \cite{winkler2003stochastic} and \cite{winkler2006stochastic} (Explicit SDAEs will be simply called SDAEs as only explicit SDAEs are considered in this article). $U$ is called the \textit{\textbf{{state space}}} and $V$ the \textit{\textbf{{algebraic space}}}. An SDAE is said to be \textit{{\textbf{autonomous (non-autonomous)}}} if both SDE and the algebraic equation are autonomous (non-autonomous). A non-autonomous system can be converted to an autonomous system by coupling the system with another differential equation $dy = dt$. Therefore, it is enough to consider an autonomous SDAE. The dimension of the algebraic equation is also called \textit{\textbf{the number of algebraic equations}} and the dimension of the algebraic space is also called \textit{\textbf{the number of algebraic variables}}. Apparently, the number of algebraic equations in equation \eqref{eq:SDAE1} is $p$ and the number of algebraic variables is $m$.

If there exist continuous stochastic processes $x(\cdot,\omega):I \to U$ and $u(\cdot,\omega): I \to V$, where $I = [0,a)$ and $0<a<\infty$, that are adapted to $\{\mathcal{F}_t\}$ and satisfy the SDE as well as the algebraic equation in the mean square sense (as in case of stochastic differential equations), then the SDAE is said to have a \textbf{\textit{local solution}} upto the random time of exiting $U\times V$. In this article the term \textit{solution} should be considered as {local solution} upto the random time of exiting $U\times V$, unless explicitly specified. For convenience, a stochastic process $y(t,\omega)$ will be simply written as $y(t)$.

A usual way of computing the solution of deterministic DAE is by differentiating the constraint and then computing the algebraic variable. In case of Ito SDEs, we will Ito differentiate the constraint. Suppose we assume that there is a solution for equations \eqref{eq:SDAE1} and the algebraic variable is given by the Ito stochastic differential equation,
\begin{equation}
du = a(x,u)dt + B(x,u)dW_t
\end{equation}
such that $a:U\times V \to \mathbb{R}^m$ and $B:U\times V \to L(\mathbb{R}^d,\mathbb{R}^m)$ are Lipschitz continuous. Assuming that the required differentiability conditions are satisfied, the Ito differentiation of the algebraic equation \eqref{eq:SDAE-AE} gives,
\begin{multline}
0 = \left[\mathcal{D}_x(g) f + \mathcal{D}_u(g) a + \dfrac{1}{2}Tr(\sigma \mathcal{D}^2_{x,x}(g)\sigma^T + B \mathcal{D}^2_{u,u}(g)B^T+ 2\sigma\mathcal{D}^2_{x,u}(g)B^T)\right] dt\\
+ \left[ \mathcal{D}_x(g)\sigma + \mathcal{D}_u(g)B + \Gamma\right] dW_t.
\end{multline}
With 
\begin{equation}
B = - (\mathcal{D}_u(g))^{-1}(\mathcal{D}_x(g)\sigma+\Gamma)
\end{equation}
and 
\begin{equation}
a = - (\mathcal{D}_u(g))^{-1}(\mathcal{D}_x(g) f + 1/2Tr(\sigma \mathcal{D}^2_{x,x}(g)\sigma^T+ B \mathcal{D}^2_{u,u}(g)B^T + 2\sigma\mathcal{D}^2_{x,u}(g)B^T))
\end{equation}
we get a stochastic curve, in the product space $U\times V$, such that the differential evolution of the algebraic equation is 0. This means that if the algebraic equation is satisfied at $t=0$, then it will be satisfied for all time as long as $\mathcal{D}_ug(x,u)$ is invertible. This type of solution has been discussed in \cite{cong2010stochastic}. We summarize this result in the following lemma.
\begin{lemma}
\label{lem:SDEforSDAE}
For an SDAE in the form of equation \eqref{eq:SDAE1} with m=p, such that $f(x,u)$ and $\sigma(x,u)$ are Lipschitz, and $g(x,u)$ is $\mathcal{C}^3$; if $\mathcal{D}_u(g)(x,u)$ is invertible and has bounded inverse in some neighbourhood of  $(x(0),u(0))$, then the local solution is given by the following SDE.
\begin{equation}
\label{eq:SDE-SDAE1-sol}
dk = (f(k),a(k))dt + (\sigma(k),B(k))dW_t, 
\end{equation} 
with $k = (x,u)$, \[B = - (\mathcal{D}_ug)^{-1}((\mathcal{D}_xg)\sigma+\Gamma),\] and \[a = - (\mathcal{D}_ug)^{-1}((\mathcal{D}_xg) f + \dfrac{1}{2}Tr(\sigma (\mathcal{D}^2_{x,x}g)\sigma^T + B (\mathcal{D}^2_{u,u}g)B^T + 2\sigma(\mathcal{D}^2_{x,u}g)B^T)).\]
\end{lemma}

\section{Index of SDAE}
\label{sec:preliminary results}
From lemma \ref{lem:SDEforSDAE}, it is clear that if $\mathcal{D}_ug$ is not invertible then we can not obtain a solution by simply Ito differentiating the constraint. Therefore, SDAEs can be categorized on the basis of invertibility of $\mathcal{D}_ug$.
\begin{definition}
An SDAE in form of equation \eqref{eq:SDAE1} will be called an \textbf{\textit{index 1 SDAE}}, if $\mathcal{D}_u(g)(x,u)$ is invertible. 
\end{definition}
\noindent One of the cases in which $\mathcal{D}_u(g)(x,u)$ is not invertible is when the algebraic equation does not admit any algebraic variable, i.e. when $g$ and $\Gamma$ are not functions of $u$.
\begin{definition}
\label{def:high_index}
An SDAE that does not admit algebraic variable in the algebraic equation (i.e. when $g$ and $\Gamma$ are not functions of the algebraic variable) will be called  {\textbf{completely high index}} SDAE. A generic form for completely high index SDAEs is given by,
\begin{subequations}
\label{eq:SDAE2}
\begin{equation}
\label{eq:SDAE2-SDE}
dx = f(x,u)dt + \sigma(x,u)dW_t,  
\end{equation}
\begin{equation}
\label{eq:SDAE2-AE}
   0 = g(x) + \int_0^t\Gamma(x(s))dW_s;
\end{equation}
\end{subequations}
with $x(t)\in \mathbb{R}^n$, $u(t)\in \mathbb{R}^m$, $W_t$ as d-dimensional Wiener process, $f:U\times V\to \mathbb{R}^n$, $\sigma:U\times V\to L(\mathbb{R}^d,\mathbb{R}^n)$, $g:U\to \mathbb{R}^p$ and $\Gamma:U \to L(\mathbb{R}^d,\mathbb{R}^p)$.
\end{definition}
\noindent For index 1 SDAEs, we know that the algebraic variables can be computed by equating the differential evolution of the algebraic equation to zero (lemma \ref{lem:SDEforSDAE}). If we use a similar approach for completely high index SDAEs then the Ito differentiation of equation \eqref{eq:SDAE2-AE} yields,
\begin{multline}
\label{equation:ito_diff_constraint_high_index}
0 = [(\mathcal{D}_xg(x_t))f(x_t,u_t)+ \dfrac{1}{2}Tr(\sigma(x_t,u_t)(\mathcal{D}^2_{x,x}g(x_t))\sigma^T(x_t,u_t))]dt \\ + [(\mathcal{D}_xg(x_t))\sigma(x_t,u_t) + \Gamma(x_t)]dW_t.
\end{multline}
We immediately notice that in order to make the differential evolution zero, both the martingale and the drift must be zero. Hence, when the number of variables is the same as the number of equations, then there is no way to find an exact solution for completely high index SDEs. However, if we have more number of algebraic variables then there might be a possibility of finding an exact solution. This is exactly why we have not put any restriction on the number of algebraic variables and the number of algebraic equations in the definition of completely high index SDAEs. 

\begin{lemma}
\label{lem:existence_SDAE2}
Consider a completely high index SDAE in the form of equation \eqref{eq:SDAE2} such that $f(x,u)$, $\sigma(x,u)$ and $\Gamma(x)$ are as smooth as required. Assume that the SDAE has a solution. Consider the following new system of SDAE,
\begin{subequations}
\begin{equation}
dx = f(x,u)dt + \sigma(x,u)dW_t,  
\end{equation}
\begin{equation}
h(x,u) = 0;
\end{equation}
\end{subequations}
where
\begin{center}
$h(x,u)= \begin{bmatrix} ( D_x(g)(x))f(x,u) + \dfrac{1}{2}Tr(\sigma(x,u) ( D^2_{x,x}(g)(x))\sigma^T(x,u))\\( D_x(g)(x))\sigma_1(x,u) + \Gamma_1(x) \\ \vdots \\( D_x(g)(x))\sigma_d(x,u) + \Gamma_d(x) \end{bmatrix}$,
\end{center}
$\sigma = [\sigma_1|\hdots |\sigma_d]$, and $\Gamma = [\Gamma_1|\hdots |\Gamma_d]$. Then the new SDAE has the same solution as that of the given SDAE.
\end{lemma}
\begin{proof}
Ito differentiation of the algebraic equation gives
\[0 = \left[ D_x(g) f(x,u) + \dfrac{1}{2}Tr(\sigma  D^2_{x,x}(g)\sigma^T) \right] dt + \left[  D_x(g)\sigma(x,u) + \Gamma\right] dW_t.\]
$h(x,u) = 0$ makes both the drift and the noise term in the above equation zero.
\end{proof}
If the new SDAE obtained in lemma \ref{lem:existence_SDAE2} is of index 1, then the Ito stochastic solution can be found using lemma \ref{lem:SDEforSDAE} (provided the conditions in lemma \ref{lem:SDEforSDAE} are satisfied). However, this comes at a cost of having more algebraic variables than the number of algebraic equations. If a completely high index SDAE in form of equation \eqref{eq:SDAE2} can be reduced to another SDAE of index 1 using the method proposed in lemma \ref{lem:existence_SDAE2}, then we will say that the SDAE  is of {\textit{Index 2}}. Therefore, for index 2 equations, the relation between the number of algebraic variables $(m)$ and number of algebraic equations $(p)$ is given by $m = p(1+d)$; where $d$ is the dimension of the Wiener process. In general, if we have to apply lemma \ref{lem:existence_SDAE2} $l$ times to ensure that the resulting SDAE is of index 1, then the relation between the number of algebraic variables ($m$) and the number of algebraic equations ($p$) must be $m = p(1+d)^{l}$, where $d$ is the dimension of the Wiener process. We will call $l+1$ as the index of an SDAE.
\begin{definition}
Let $m$ be the number of algebraic variables, $p$ be the number of algebraic equations, and $d$ be the dimension of the Wiener process. An SDAE will be said to be \textbf{\textit{indexed}} if there exists a positive integer $J$ such that
\[J = 1+\dfrac{log(m) - log(p)}{log(1+d)}.\]
Furthermore, if $(J-1)$ times repeated application of lemma \ref{lem:existence_SDAE2} gives an index 1 SDAE, then $J$ will be called the \textbf{\textit{index}} of the SDAE.
\end{definition}
As the dimension of spaces is not available by choice, not every SDAE can be indexed. Moreover, in many real world problems, we usually encounter completely high index equations in which $m = p$. 

\section{Existence and uniqueness of the solution}
\label{sec:n&Scondition}
Using the definitions from the previous section, we see that the trivial method from lemma \ref{lem:SDEforSDAE} cannot be used for completely high index equations. Therefore, before constructing methods for computation, it is important to obtain necessary and sufficient conditions for the existence of a solution for SDAEs. Moreover, we shall see in section \ref{sec:necessary condition}, that the necessary condition is required to prove the existence of a class of completely high index problems called \textit{ill-posed} equations which have no Ito stochastic solution.
\subsection{Sufficient condition for existence of unique solution}
\label{sec:existence_and_uniq}
The proof of existence of unique solution of a Stochastic Differential Equation depends on Picard-type iteration or contraction mapping theorem in an appropriate metric space. We will use the idea of contraction mapping to prove the existence of a unique solution of an SDAE given by equation \eqref{eq:SDAE1}. For this, let $(\Omega, \mathcal{F},\{\mathfrak{F}_t\}, P)$ be the filtered probability space and the Brownian motions $W_j$ be adapted to the filtration $\mathfrak{F}$. Let $||\cdot||_c$, $||\cdot||_x$ and $||\cdot||_u$ be norms in $\mathbb{R}^{n+m}$, $\mathbb{R}^{n}$ and $\mathbb{R}^{m}$ respectively, such that $||(a,b)||^2_c = ||a||^2_x + ||b||^2_u$. Let $E_x[a,b]$ denote the Banach space $(L^2_\mathfrak{F}(\Omega;\mathcal{C}([a,b],\mathbb{R}^n)), (\mathbb{E}\sup_{t\in [a,b]}\left|\left|\cdot\right|\right|^2_x)^{1/2})$ i.e. the space of all continuous $\mathfrak{F}$-adapted processes from $[a,b]$ to $\mathbb{R}^n$. Similarly, let $E_u[a,b]$ denote the space $(L^2_\mathfrak{F}(\Omega;\mathcal{C}([a,b],\mathbb{R}^m)), (\mathbb{E}\sup_{t\in [a,b]}\left|\left|\cdot\right|\right|^2_u)^{1/2})$. Consider a function $\phi: E_x[0,a]\times E_u[0,a]\to E_x[0,a]\times E_u[0,a]$ such that
\begin{equation}
\label{equation:inter_sufficient_0}
\phi(x_t,u_t) = \begin{bmatrix} x_0 + \int_0^tf(x(s),u(s))ds\\ u(t) - g(x(t),u(t))\end{bmatrix} + \begin{bmatrix}\int_0^t\sigma(x(s),u(s))dW_s\\  - \int_0^t \Gamma(x(s),u(s))dW_s\end{bmatrix},
\end{equation}
If $\phi$ is a contraction map then the fixed point of $\phi$ is the local solution for the SDAE given by equation \eqref{eq:SDAE1}. Hence, the problem has reduced to proving that $\phi$ is indeed a contraction map.  In the  proof of the following statement we prove that if we consider  $\sup\{||\mathcal{D}(h(x,u))||, (x,u)\in B_{\epsilon_x}(x_0)\times B_{\epsilon_u}(u_0)\} < \dfrac{1}{\sqrt{2}}$, then $\phi$ becomes a contraction map.
\begin{proposition}
\label{proposition:existence_SDAE1}
Consider the SDAE in the form of equation \eqref{eq:SDAE1} such that the number of algebraic variables ($m$) is equal to the number of algebraic equations ($p$), i.e. $m = p$. Let $f(x,u)$, $\sigma(x,u)$, and $\Gamma(x,u)$ be locally Lipschitz in $B_{\epsilon_x}(x_0)\times B_{\epsilon_u}(u_0)$ for some $u_0\in V$ and $x_0\in U$. Suppose $h(x,u) = (0,u - g(x,u))$ such that $h$ is locally Lipschitz in $B_{\epsilon_x}(x_0)\times B_{\epsilon_u}(u_0)\}$ with constant $M$ such that \[M<\dfrac{1}{\sqrt{2}}.\] Suppose $g(x_0,u_0) = 0$. Then there exists a unique local solution for the equation, i.e. there exists some $a>0$ such that the solution exists for all $t\in [0,min(a,\tau(\omega)))$, where $\tau(\omega) = \inf\{t\in \mathbb{R}^+|(x(t,\omega),u(t,\omega))\notin B_{\epsilon_x}(x(0))\times B_{\epsilon_u}(u(0))\}$.
\end{proposition}
\begin{proof}
In the preceding discussion, we have observed that if we consider $\phi(x_t,u_t)$ as given in equation \eqref{equation:inter_sufficient_0}, then it is enough to prove that $\phi$ is a contraction map. However, as the Lipschitz condition on the coefficients $f(x,u)$, $\sigma(x,u)$, and $\Gamma(x,u)$ is local in $B_{\epsilon_x}(x(0))\times B_{\epsilon_u}(u(0))$, we will consider processes that are stopped on exiting $B_{\epsilon_x}(x(0))\times B_{\epsilon_u}(u(0))$. If required, the size of $\epsilon_x$ and $\epsilon_u$ can be reduced to ensure that $\mathbb{E}\sup_{t\in [0,a]}||(x_t,u_t)||_c<\infty$ in $B_{\epsilon_x}(x(0))\times B_{\epsilon_u}(u(0))$. This ensures that the function $\phi$ indeed maps to the Banach space $E_x[0,a]\times E_u[0,a]$. Therefore, the modified function $\phi: E_x[0,a]\times E_u[0,a]\to E_x[0,a]\times E_u[0,a]$ is given as
\begin{equation}
\label{equation:inter_sufficient_1}
\phi(x_t,u_t) = \begin{bmatrix} x_0 + \int_0^{t\wedge\tau}f(x(s),u(s))ds\\ u(t) - g(x(t),u(t))\end{bmatrix} + \begin{bmatrix}\int_0^{t\wedge\tau}\sigma(x(s),u(s))dW_s\\  - \int_0^{t\wedge\tau} \Gamma(x(s),u(s))dW_s\end{bmatrix},
\end{equation}
where $\tau(\omega) = \inf\{t\in I|(x(t,\omega),u(t,\omega))\notin B_{\epsilon_x}(x(0))\times B_{\epsilon_u}(u(0))\}$. The fixed point of $\phi$ gives the local solution for the SDAE in form of equation \eqref{eq:SDAE1}. Hence, we need to prove that $\phi$ is indeed a contraction map.

As $f$, $\sigma$ and $\Gamma$ are locally Lipschitz, we consider $(x(t),u(t))\in B_{\epsilon_x}(x(0))\times B_{\epsilon_u}(u(0))\subset U\times V$ such that $f$, $\sigma$ and $\Gamma$ are Lipschitz with constants $k_f$, $k_\sigma$ and $k_\Gamma$ respectively. It is assumed that $g(x(0),u(0)) = 0 $. If $y_t(\omega),z_t(\omega) \in E_x[0,a]$, and $v_t(\omega), r_t(\omega)\in E_u[0,a]$, then for fixed sample path we get
\begin{multline}
\label{equation:inter_sufficient_3}
||\phi(y_t,v_t) - \phi(z_t,r_t)||^2_c \leq 2\left|\left|\int_0^{t\wedge\tau}f(y(s),v(s)) - f(z(s),r(s))ds\right|\right|^2_x \\
+2\left|\left|\int_0^{t\wedge\tau} \sigma(y(s),v(s)) - \sigma(z(s),r(s))dW_s\right|\right|^2_x+ 2\left|\left| v(t) - r(t) - (g(y(t),v(t)) - g(z(t),r(t)))\right|\right|_u^2\\
+2\left|\left| \int_0^{t\wedge\tau} \Gamma(y(s),v(s)) -\Gamma(z(s),r(s))dW_s\right|\right|_u^2.
\end{multline}
Using Jensen's inequality in the first term,
\begin{multline}
{||\phi(y_t,v_t) - \phi(z_t,r_t)||^2_c}
\leq 2k_f^2\left(\int_0^{t\wedge\tau}||(y(s),v(s)) - (z(s),r(s))||_cds\right)^2\\
+ 2\left|\left|\int_0^{t\wedge\tau} \sigma(y(s),v(s)) - \sigma(z(s),r(s))dW_s\right|\right|^2_x
+ 2\left|\left| v(t) - r(t) - (g(y(t),v(t)) - g(z(t),r(t)))\right|\right|_u^2\\ +2\left|\left| \int_0^{t\wedge\tau} \Gamma(y(s),v(s)) - \Gamma(z(s),r(s))dW_s\right|\right|_u^2.
\end{multline}
\begin{multline}
\implies {||\phi(y_t,v_t) - \phi(z_t,r_t)||^2_c} \leq 2a^2k_f^2\sup_{t\in [0,a]}||(y(t),v(t)) - (z(t),r(t))||^2_c\\
+ 2\sup_{t\in [0,a]}\left|\left|\int_0^{t\wedge\tau} \sigma(y(s),v(s)) - \sigma(z(s),r(s))dW_s\right|\right|^2_x\\
+ 2\sup_{t\in [0,a]}\left|\left| v(t) - r(t) - (g(y(t),v(t)) - g(z(t),r(t)))\right|\right|_u^2\\
+2\sup_{t\in [0,a]}\left|\left| \int_0^{t\wedge\tau} \Gamma(y(s),v(s)) - \Gamma(z(s),r(s))dW_s\right|\right|_u^2.
\end{multline}
\begin{multline}
\implies \sup_{t\in [0,a]}{||\phi(y_t,v_t) - \phi(z_t,r_t)||^2_c}\leq 2a^2k_f^2\sup_{t\in [0,a]}||(y(t),v(t)) - (z(t),r(t))||^2_c\\
+ 2\sup_{t\in [0,a]}\left|\left|\int_0^{t\wedge\tau} \sigma(y(s),v(s)) - \sigma(z(s),r(s))dW_s\right|\right|^2_x\\
+ 2\sup_{t\in [0,a]}\left|\left| v(t) - r(t) - (g(y(t),v(t)) - g(z(t),r(t)))\right|\right|_u^2\\
+2\sup_{t\in [0,a]}\left|\left| \int_0^{t\wedge\tau} \Gamma(y(s),v(s)) - \Gamma(z(s),r(s))dW_s\right|\right|_u^2.
\end{multline}
Since this is true for every sample path,
\begin{multline}
\mathbb{E}\sup_{t\in [0,a]}{||\phi(y_t,v_t) - \phi(z_t,r_t)||^2_c} \leq 2a^2k_f^2\mathbb{E}\sup_{t\in [0,a]}||(y(t),v(t)) - (z(t),r(t))||^2_c\\
+ 2\mathbb{E}\sup_{t\in [0,a]}\left|\left|\int_0^{t\wedge\tau} \sigma(y(s),v(s)) - \sigma(z(s),r(s))dW_s\right|\right|^2_x\\
+ 2\mathbb{E}\sup_{t\in [0,a]}\left|\left| v(t) - r(t) - (g(y(t),v(t)) - g(z(t),r(t)))\right|\right|_u^2\\
+2\mathbb{E}\sup_{t\in [0,a]}\left|\left|  \int_0^{t\wedge\tau}\Gamma(y(s),v(s)) -\Gamma(z(s),r(s))dW_s\right|\right|_u^2.
\end{multline}
Using Doob's inequality on the Ito integral terms,
\begin{multline}
\mathbb{E}\sup_{t\in [0,a]}{||\phi(y_t,v_t) - \phi(z_t,r_t)||^2_c}\leq 2a^2k_f^2\mathbb{E}\sup_{t\in [0,a]}||(y(t),v(t)) - (z(t),r(t))||^2_c\\
+8k_\sigma^2\mathbb{E}\int_0^{t\wedge\tau} ||(y(s),v(s)) - (z(s),r(s))||^2_cds\\
+ 2\mathbb{E}\sup_{t\in [0,a]}\left|\left| v(t) - r(t) - (g(y(t),v(t)) - g(z(t),r(t)))\right|\right|_u^2\\
+8k_\Gamma^2 \mathbb{E}\int_0^{t\wedge\tau} ||(y(s),v(s)) - (z(s),r(s))||^2_cds.
\end{multline}
\begin{multline}
\label{equation:inter_sufficient10}
\implies \mathbb{E}\sup_{t\in [0,a]}{||\phi(y_t,v_t) - \phi(z_t,r_t)||^2_c}\\
\leq ((2k^2_f) a^2  + 8(k^2_\sigma + k^2_\Gamma)a)\mathbb{E}\sup_{t\in [0,a]}||(y(t),v(t)) - (z(t),r(t))||^2_c\\
+ 2\mathbb{E}\sup_{t\in [0,a]}\left|\left| v(t) - r(t) - (g(y(t),v(t)) - g(z(t),r(t)))\right|\right|_u^2.
\end{multline}

Since $M$ is the Lipschitz constant of $h(x,u) = (0,u-g(x,u))$,
\[||(0,(v-r) - (g(y,v) - g(z,r)))||_c\leq M||(y-z,v-r)||_c.\]
Therefore,
\[\mathbb{E}\sup_{t\in [0,a]}\left|\left| v(t) - r(t) - (g(y(t),v(t)) - g(z(t),r(t)))\right|\right|_u^2\]\[\leq M^2\mathbb{E}\sup_{t\in [0,a]}||(y(t),v(t)) - (z(t),r(t))||^2_c.\]
On substituting this in equation  \eqref{equation:inter_sufficient10} we get,
\begin{multline}
\mathbb{E}\sup_{t\in [0,a]}{||\phi(y_t,v_t) - \phi(z_t,r_t)||^2_c}\\
\leq( (2k^2_f) a^2  + 8(k^2_\sigma + k^2_\Gamma)a +2M^2)\mathbb{E}\sup_{t\in [0,a]}||(y(t),v(t)) - (z(t),r(t))||^2_c.
\end{multline}
To make $\phi$ a contraction map, we must choose $a>0$ such that 
\begin{equation}
\label{equation:intermediate_sufficeint11}
(2k^2_f) a^2  + 8(k^2_\sigma + k^2_\Gamma)a+ 2M^2\leq 1.
\end{equation}
As we are given that $M<1/\sqrt{2}$, we can find small enough $a>0$ such that it satisfies equation  \eqref{equation:intermediate_sufficeint11}. With this small enough $a>0$, $\phi$ becomes a contraction map. Hence, we have a local solution for the SDAE upto time $min(a,\tau(\omega))$.
\end{proof}
\begin{remark}
{If the solution does not exist or is not unique then the conditions of the proposition are violated. However, as the proposition is not an ``if and only if" statement, the converse may not be true.}
\end{remark}
If we assume that the function $g$ is $\mathcal{C}^1$, then we find that
\[M = \sup_{(x,u)\in B_{\epsilon_x}(x_0)\times B_{\epsilon_u}(u_0)} \left|\left|\begin{matrix}
0\qquad \quad 0\\
- D_xg \quad I -  D_ug
\end{matrix}\right|\right|.\]
In this case, it can be seen that if $g$ is not a function of $u$, then $D_ug = 0$ and hence the condition in the proposition is violated; i.e. in case of completely high index SDAEs wherein the constraint is not a function of the algebraic variable, the condition in proposition \ref{proposition:existence_SDAE1} is always violated.  But, as the violation of the condition does not amount to non-existence of unique solution, a solution may or may not exist for completely high index SDAE. Therefore, proposition \ref{proposition:existence_SDAE1} is of no use in case of completely high index problems with $m = p$. Hence, a necessary condition is required.

\subsection{Necessary condition}
\label{sec:necessary condition}
To find a necessary condition for existence of a solution, we assume that there exists a solution for SDAE equation \eqref{eq:SDAE1}
\begin{subequations}
\begin{equation}
dx = f(x,u)dt + \sigma(x,u)dW_t, 
\tag{\ref{eq:SDAE-SDE}}
\end{equation}
\begin{equation}
g(x,u) = - \int_0^t\Gamma(x(s),u(s))dW_s;
\tag{\ref{eq:SDAE-AE}}
\end{equation}
\end{subequations}
such that the algebraic variable is given by $du = a(x,u)dt + B(x,u)dW_t$.

To recall, $W_t$ is a d-dimensional Wiener process, $f:U\times V\to \mathbb{R}^n$, $g:U\times V\to \mathbb{R}^p$, $\sigma:U\times V\to L(\mathbb{R}^d,\mathbb{R}^n)$, and $\Gamma:U\times V\to L(\mathbb{R}^d,\mathbb{R}^p)$.

For simplicity, let us assume that the constraint is noiseless i.e. $\Gamma =  0$. Ito differentiation of the constraint gives
\begin{multline}
0 = \left[ D_x(g) f +  D_u(g) a\right] dt + \left[  D_x(g)\sigma +  D_u(g)B\right] dW_t\\
+\left[ \dfrac{1}{2}Tr(\sigma  D^2_{x,x}(g)\sigma^T + B  D^2_{u,u}(g)B^T+ 2\sigma D^2_{x,u}(g)B^T)\right] dt.
\end{multline}
This means that both the drift and the noise is zero. In particular, the noise coefficient being zero implies that \[(\sigma_i(x(t),u(t)), B_i(x(t),u(t)))\in Ker( D_{(x(t),u(t))}(g)),\]
where $(\sigma_i(x(t),u(t)), B_i(x(t),u(t)))$ is the $i^{th}$ column vector of the matrix $\begin{bmatrix}\sigma(x,u) \\ B(x,u)\end{bmatrix}$.
\begin{lemma}
\label{lemma:subbundle}
Consider SDAE given in form of equation \eqref{eq:SDAE1}. Let $\Gamma = 0$. Suppose there exists Ito stochastic solutions for any initial condition that also satisfies the constraint such that the algebraic variable is given by the following SDE,
\[du = a(x,u)dt + B(x,u)dW_t.\]
If $g(x,u)$ is $\mathcal{C}^3$ and all the vector fields are $\mathcal{C}^1$, then
\[Img(\sigma(k),B(k))\subset T_kS,\]
where $S = \{g^{-1}(0)\}$ is a submanifold of $U\times V$.
Moreover, if $Img(\sigma(k),B(k))$ is full rank for all $k\in S$, then 
\[\underset{k\in S}{\bigcup} Img(\sigma(k),B(k))\] will be a distribution on S; where $Img(\sigma(k),B(k))$ is the column space of $\begin{bmatrix}\sigma(x,u) \\ B(x,u)\end{bmatrix}$.
\end{lemma}
\begin{proof}
Let $k = (x,u)$. This results in
  \[dk =\begin{bmatrix}f(k) \\ a(k)\end{bmatrix}dt + \begin{bmatrix}\sigma(k)\\ B(k)\end{bmatrix}dW_t.\]
Using the submersion theorem from differential geometry, if $0$ is a regular value of $g$ then the level set $S = \{g^{-1}(0)\}$ is a submanifold of $U\times V \subset \mathbb{R}^n\times\mathbb{R}^m$ with tangent space given by $T_{k}S = ker(Dg(k))$.\\
From the Ito differentiation formula, it is clear that if $k(t,\omega)$ is the solution of the given SDAE then
\[D_xg(k(t,\omega)) \cdot \sigma_i(k(t,\omega)) + D_ug(k(t,\omega)) \cdot B_i(k(t,\omega)) = 0\]
$\forall \quad i\in\{1,\hdots ,d\}$.
Since
\[T_kS = ker(Dg(k))= ker([D_xg(x,u) \quad D_ug(x,u)]),\] it can be concluded that for a given sample path,
\[(\sigma_i(k(t,\omega))), B_i(k(t,\omega))) \in T_{k(t,\omega)}S \quad \forall \quad i \in\{1,\hdots ,d\}.\]
But this is true for all sample paths (i.e. $\forall\: \omega\in\Omega$). So the sample paths for any initial condition $k(0)\in S$ implies that $k(t,\omega)\in S$. Therefore,
\[ Img\begin{bmatrix}\sigma(k) \\ B(k)\end{bmatrix}\subset T_{k}S \: \forall \: k\in S.\]
Moreover when $Img(\sigma(k),B(k))$ is full rank for all $k\in S$, $\underset{k\in S}{\bigcup} Img(\sigma(k),B(k))$ is nothing but the kernel distribution on $S$.
\end{proof}

This statement can be extended to equations with noisy constraint as well. However, the SDAE needs to be tweaked. 
An SDAE with noisy constraint (equation \eqref{eq:SDAE1}) can be expressed as an SDAE with noiseless constraint by substituting
\[z(t) = \int_0^t\Gamma(x(s), u(s))dW_s.\]
With this,
\begin{subequations}
\label{eq:SDAE1_new}
\begin{equation}
\label{eq:SDAE-SDE_new}
dx = f(x,u)dt + \sigma(x,u)dW_t, 
\end{equation}
\begin{equation}
\label{eq:SDAE-SDE_new2}
dz =0dt + \Gamma(x,u)dW_t,
\end{equation}
\begin{equation}
\label{eq:SDAE-AE_new}
h(x,z,u) = g(x,u)+ z = 0;
\end{equation}
\end{subequations}
and we get a new SDAE with noiseless constraint of higher dimensional state space. This idea is similar in spirit to the method of suspension used in the theory of ordinary differential equations.

The solution of the original SDAE will match the solution of equation \eqref{eq:SDAE1_new} if the initial condition for $z$ is $z(0) = 0$. It should be noted that while there is freedom to choose the initial condition for $x$ and $u$ (as long as the initial condition for $x$ and $u$ satisfies the constraint), the initial condition for $z$ must be $0$. Therefore, a geometric result that is similar to lemma \ref{lemma:subbundle} can not be obtained by direct application of the lemma to equation \eqref{eq:SDAE1_new}. However, if the solution exists, then it will satisfy the constraint $h(x,z,u) = 0$. In such a case, following the arguments in proof of the lemma \ref{lemma:subbundle}, we find that $(\sigma_i(x(t,\omega),u(t,\omega)), \Gamma_i(x(t,\omega),u(t,\omega)), B_i(x(t,\omega),u(t,\omega)))$ $\in$ $T_{(x(t,\omega),z(t,\omega),u(t,\omega))}S$ $\forall$ $i \in\{1,\hdots ,d\},$ where $S = h^{-1}(0)$ is a submanifold of $U\times\mathbb{R}^m\times V$. In other words,
\[ Img\begin{bmatrix}\sigma (x(t,\omega),u(t,\omega)) \\ \Gamma(x(t,\omega),u(t,\omega)) \\ B(x(t,\omega),u(t,\omega))\end{bmatrix}\subset T_{(x(t,\omega),z(t,\omega),u(t,\omega))}S.\]
Therefore, we can state the following lemma for SDAEs with noisy constraints.
\begin{lemma}
\label{lemma:noisy_to_noiseless}
Consider SDAE given in form of equation (\ref{eq:SDAE1}). Suppose there exists an Ito stochastic solution such that the algebraic variable is given by the following SDE
\[du = a(x,u)dt + B(x,u)dW_t.\]
If $g(x,u)$ is $\mathcal{C}^3$ and all the vector fields are $\mathcal{C}^1$, then
\[ Img\begin{bmatrix}\sigma (x(t,\omega),u(t,\omega)) \\ \Gamma(x(t,\omega),u(t,\omega)) \\ B(x(t,\omega),u(t,\omega))\end{bmatrix}\subset T_{(x(t,\omega),z(t,\omega),u(t,\omega))}S.\]
Here $S = h^{-1}(0)$ is a submanifold of $U \times \mathbb{R}^p\times V$, $h(x,z,u) = g(x,u)+ z$, and $z(t) = \int_0^t\Gamma(x(s), u(s))dW_s$.
\end{lemma}

\begin{corollary}
\label{corollary:subbundle_euclidean}
For completely high index SDAE given by equation \eqref{eq:SDAE2}
\begin{subequations}
\begin{equation}
dx = f(x,u)dt + \sigma(x,u)dW_t,
\tag{\ref{eq:SDAE2-SDE}}
\end{equation}
\begin{equation}
0 = g(x) + \int_0^t\Gamma(x_s)dW_s;
\tag{\ref{eq:SDAE2-AE}}
\end{equation}
\end{subequations}
let $g(x)$ be $\mathcal{C}^3$ and all the vector fields be $\mathcal{C}^1$.  Assume that there exists a solution for the SDAE. Then
\[ Img\begin{bmatrix}\sigma (x(t,\omega),u(t,\omega)) \\ \Gamma(x(t,\omega))\end{bmatrix}\subset T_{(x(t,\omega),z(t,\omega))}S.\]
Here $S = h^{-1}(0)$ is a submanifold of $U \times \mathbb{R}^p$, $h(x,z) = g(x)+ z$, and $z(t) = \int_0^t\Gamma(x(s))dW_s$.
\end{corollary}
\begin{definition}
Consider completely high index SDAEs that do not admit algebraic variable in the noise ($\sigma$ is not a function of $u$) and given by the following equation
\begin{subequations}
\label{equation:UNSDAE}
\begin{equation}
dx = f(x,u)dt + \sigma(x)dW_t,
\end{equation}
\begin{equation}
g(x) = -\int_0^t\Gamma(x(s))dW_s;
\end{equation}
\end{subequations}
where $x(t)\in\mathbb{R}^n$, $u(t)\in \mathbb{R}^m$, $W_t$ is a d-dimensional Weiner process, $f:U\times V\to\mathbb{R}^n$, $\sigma:U\to L(\mathbb{R}^d,\mathbb{R}^n)$, $g:U\to\mathbb{R}^p$, and $\Gamma:U\to L(\mathbb{R}^d,\mathbb{R}^p)$.

If $Img(\sigma(x),\Gamma(x))$ is not in $ker[D_xg(x)\quad  I_{p\times p}]$, then the SDAE \eqref{equation:UNSDAE} will be called an \textbf{ill-posed SDAE}.
\end{definition}
\begin{corollary}
\label{corollary:no_sol_to_UNSAE}
Consider an ill-posed SDAE in form of equation \eqref{equation:UNSDAE} such that $g(x)$ is $\mathcal{C}^3$ and all the other vector fields are $\mathcal{C}^1$. Then there is no solution for equation \eqref{equation:UNSDAE}.
\end{corollary}
\begin{proof}
Apply contrapositive statement of corollary \ref{corollary:subbundle_euclidean} to UNSDAE.
\end{proof}

\section{Approximate Ito stochastic solution for completely high index SDAEs}
In section \ref{sec:preliminary results} we have observed that, if the SDAE is driven by a $d$-dimensional Wiener process and if the relation between the number of algebraic variables ($m$) and the number of algebraic equations ($p$)  is given by $m = p(1+d)^{J-1}$ for some positive integer $J$, then it is possible to find exact solution. However for completely high index SDAEs if $m = p$, then we have observed that we have no traditional ways of obtaining an exact solution. In section \ref{sec:necessary condition}, using the necessary condition we have proved the existence of a class of SDAEs for which there is no solution. We have called these equations as ill-posed SDAEs. From an application point of view, it is always desirable to have solution for ill-posed problems as well. However, the existence of the ill-posed SDAEs makes it challenging to develop a general technique for solving all types of completely high index SDAE. In this section, we develop some techniques to find an approximate Ito stochastic solution for all types of completely high index SDAEs. 
\subsection{Approximate solution with unit probability}
\label{sec:gen_ito}
Consider the generic form of completely high index SDAE given by equation \eqref{eq:SDAE2}
\begin{subequations}
\begin{equation}
dx = f(x,u)dt + \sigma(x,u)dW_t,
\tag{\ref{eq:SDAE2-SDE}}
\end{equation}
\begin{equation}
0 = g(x) + \int_0^t\Gamma(x(s))dW_s;
\tag{\ref{eq:SDAE2-AE}}
\end{equation}
\end{subequations}
with state space $U = \mathbb{R}^n$, and algebraic space $V = \mathbb{R}^m$. As observed in the discussion preceding lemma \ref{lemma:noisy_to_noiseless}, an SDAE with noisy constraint can be expressed as another SDAE with noiseless constraint. For this reason $\Gamma = 0$ can be considered without loss of generality. Therefore the equation under consideration becomes
	\begin{subequations}
	\label{eq:SDAE2_genito}
	\begin{equation}
	dx = f(x,u)dt + \sigma(x,u)dW_t,  
	\end{equation}
	\begin{equation}
    0 = g(x);
	\end{equation}
	\end{subequations}
	with $x(t)\in \mathbb{R}^n$, $u(t)\in \mathbb{R}^m$, $W_t$ as d-dimensional Wiener process, $f:\mathbb{R}^n\times \mathbb{R}^m\to \mathbb{R}^n$, $\sigma:\mathbb{R}^n\times \mathbb{R}^m\to L(\mathbb{R}^d,\mathbb{R}^n)$, and $g:\mathbb{R}^n\to \mathbb{R}^p$. 
\begin{definition}
If, for some $\epsilon>0$, there exists an Ito stochastic process given by $du(t,\omega) = a(x,u) dt + B(x,u)dW_t$ such that the solution of the differential equation \[d\left[\begin{matrix}x\\ u\end{matrix}\right] = \left[\begin{matrix}f(x,u)\\ a(x,u)\end{matrix}\right]dt + \left[\begin{matrix}\sigma(x,u) \\ B(x,u) \end{matrix}\right]dW_t,\]
satisfies $||g(x(t))||<\epsilon$ then the corresponding stochastic process $(x(t),u(t))$ is said to be an \textbf{approximate solution with unit probability} for $\epsilon$.
\end{definition}
With this definition, the problem of finding approximate solution boils down to finding $a(x,u)$ and $B(x,u)$ for given $\epsilon>0$.\\ Suppose there exists a function $y:\mathbb{R}\times\mathbb{R}^m\to\mathbb{R}^n$ such that
\[y(t,u(t)) = x(t).\] If $a(x,u) = 0$ and $B(x,u) = 0$ then $u(t)$ will be constant $u(t) = u_0 = v$. However, as $y(t,v) = x(t)$,  $y(t,v)$ is still stochastic.  This means that when $a(x,u) \neq 0$ and $B(x,u)\neq 0$, $y(t,u(t)) = x(t)$ is obtained by composition of the stochastic process $u(t)$ into the stochastic process $y(t,\cdot)$. Finding differential evolution of such composition is not a straight forward application of Ito differentiation. However, \textit{extended Ito formula} from \cite{kunita1981some} can be used. Extended Ito formula can be interpreted as the stochastic analogue of the method of characteristics, which is used in Partial Differential Equations (PDEs). The formula is also known by other names such as \textit{generalized Ito formula or Ito-Wentzell's formula}. Using this formula,
\begin{multline}
d(y_i(t,u(t))) = dM^i_t(u(t)) + \left(\partial_{v_j}y_i(t,u(t))a_j(y(t,u(t)),u(t))\right)dt \\ +\left( \dfrac{1}{2}B_{km}(\partial^2_{v_jv_k}y_i(t,u(t)))B_{jm}\right) dt+d\left[\dfrac{\partial M_t^i}{\partial v_l},du_l\right]\\
+ \left(\partial_{v_k}y_i(t,u(t))B_{kj}(y(t,u(t)),u(t))\right) dW^j_t;
\end{multline}
where $M_t(v) = y(t,v)$ and $dM_t(v) = dy(t,v)$ for some constant $v\in \mathbb{R}^m$.\\
As $x(t) = y(t,u(t))$, it can be verified that if
\begin{multline}
M^i_t(v) = y_i(t,v) = y_i(0,v)+\int_0^tf_i(y(s,v),v) - \partial_{v_j}y_i(s,v)a_j(y(s,v),v)ds\\
-\int_0^t \dfrac{1}{2}B_{km}(y(s,v),v)(\partial^2_{v_jv_k}y_i(s,v))B_{jm}(y(s,v),v))ds-\int_0^t \partial_{v_l}\Lambda_{ij}(s,v)B_{lj}(y(s,v),v)ds\\ + \int_0^t\Lambda_{ij}(s,v) dW^j_s,
\end{multline}
where
\[\Lambda_{ij}(t,v) = \sigma_{ij}(y(t,v),v) - \partial_{v_k}y_i(t,v)B_{kj}(y(t,v),v),\]
then
\[y(t,u(t)) = x(t) = x(0) + \int_0^tf ds + \int_0^t\sigma dW_s.\]
On Ito-differentiating the algebraic equation, keeping $v$ constant,
\begin{equation}
\label{eq:gen_Ito_g_evolv}
0= g(y(t,v)) = g(y(0,v)) + \int_0^t\left[( Dg)F+\dfrac{1}{2}Tr(\Lambda^T ( D^2g)\Lambda)\right]ds+ \int_0^t\left[( Dg)\Lambda\right]dW_s,
\end{equation}
where
\[F_i = f_i - (\partial_{v_j}y_i)a_j-\dfrac{1}{2}B_{km}(\partial^2_{v_jv_k}y_i)B_{jm}-(\partial_{v_l}\Lambda_{ij})B_{lj}\] and
\[\Lambda = \sigma -  D_vy B.\]
\begin{proposition}
Consider equation \eqref{eq:SDAE2_genito} that is given as
	\begin{subequations}
	\begin{equation}
	dx = f(x,u)dt + \sigma(x,u)dW_t,
	\tag{\ref{eq:SDAE2_genito}a}
	\end{equation}
	\begin{equation}
    0 = g(x);
    \tag{\ref{eq:SDAE2_genito}b}
	\end{equation}
	\end{subequations}
	with $x(t)\in \mathbb{R}^n$, $u(t)\in \mathbb{R}^m$, $W_t$ as d-dimensional Wiener process, $f:\mathbb{R}^n\times \mathbb{R}^m\to \mathbb{R}^n$, $\sigma:\mathbb{R}^n\times \mathbb{R}^m\to L(\mathbb{R}^d,\mathbb{R}^n)$, and $g:\mathbb{R}^n\to \mathbb{R}^p$. Let $z_t:\mathbb{R}^m\to \mathbb{R}^p$ such that
\[z_t(v) =g(y(t,v)).\]
Assume that $z_0(v)$ is invertible for all $v\in\mathbb{R}^m$. Let $\epsilon>0$ and
\begin{equation}
\label{equation:gen_ito_solution}
\begin{aligned}
B = (( Dg) D_vy)^{-1}(( Dg)\sigma), \\
a = (( Dg) D_vy)^{-1}w;
\end{aligned}
\end{equation}
where
\[w_i =(\partial_{x_u}g_i) (f_u -\dfrac{1}{2}B_{km}(\partial^2_{v_jv_k}y_u)B_{jm}-(\partial_{v_l}\Lambda_{uj})B_{lj})+ \dfrac{1}{2}\Lambda_{kj} (\partial^2_{x_k x_l}g_i)\Lambda_{lj}\] and
\[\Lambda = \sigma -  D_vy B.\]
If $y(t,v)$ is chosen such that $||z_0(v)||<\epsilon$ for all $v\in\mathbb{R}^m$, then the solution for
\[d\left[\begin{matrix}x\\ u\end{matrix}\right] = \left[\begin{matrix}f(x,u)\\ a(x,u)\end{matrix}\right]dt + \left[\begin{matrix}\sigma(x,u) \\ B(x,u) \end{matrix}\right]dW_t\]
is the local approximate solution with unit probability for given $\epsilon$.
\end{proposition}
\begin{proof}
Equation \eqref{eq:gen_Ito_g_evolv} is a stochastic PDE representing the dynamics of $z_t(v) =g(y(t,v))$.
Now, if $y(0,v)$ is chosen such that $||z_0(v)||<\epsilon$ for all $v\in\mathbb{R}^m$, then $a$ and $B$ can be estimated from equation \eqref{eq:gen_Ito_g_evolv} such that $dz_t(v) = 0$.
With $B = (( Dg) D_vy)^{-1}(( Dg)\sigma)$, and $a = (( Dg) D_vy)^{-1}w$; where
\[w_i =(\partial_{x_u}g_i) (f_u -\dfrac{1}{2}B_{km}(\partial^2_{v_jv_k}y_u)B_{jm}-(\partial_{v_l}\Lambda_{uj})B_{lj})+ \dfrac{1}{2}\Lambda_{kj} (\partial^2_{x_k x_l}g_i)\Lambda_{lj}\] and
\[\Lambda = \sigma -  D_vy B,\] the drift and the martingale part of equation \eqref{eq:gen_Ito_g_evolv} becomes zero. Hence, the dynamics of $z_t(v)$ is frozen in time i.e. $z_0(v) = z_t(v)\forall t\in I$, where $I$ is some interval of existence. Hence, one can simply talk about the function $z_0(v)$ instead of $z_t(v)$. From equation (\ref{eq:gen_Ito_g_evolv}) it is clear that estimation of $a$ and $B$ requires $ Dz_0(v)$ to be invertible (and it is already given that  $Dz_0(v)$ is invertible). Therefore, the number of algebraic variables should be equal to number of algebraic equations, i.e. $m = p$. Since, we have chosen $y(0,v)$ is chosen such that $||z_0(v)||<\epsilon$ for all $v\in\mathbb{R}^m$, the frozen dynamics given by $dz_t(v) = 0$ implies that $||z_0(u(t))|| = ||z_t(u(t))|| = ||g(y(t,u(t)))|| = ||g(x(t))||<\epsilon$ $\forall t\in I$, where $I$ is the interval of existence for
\[d(x,u) = (f(x,u),a(x,u)) dt + (\sigma(x),B(x,u))dW_t.\]
This is nothing but approximate solution with unit probability for $\epsilon$.
\end{proof}

\begin{remark} Hence, if $\epsilon\to 0$, then $||z_0(u(t))|| \to 0$ $\forall t\in I$. But, from corollary \ref{corollary:no_sol_to_UNSAE}, we know that there is no solution for ill-posed equations. Therefore, we can expect the solution to blow off. Indeed, this can be confirmed from the fact that as $\epsilon\to 0$, $|| Dz_0(v)||\to 0$ $\forall v\in \mathbb{R}^m$ and hence the invertibility of $Dz_0(v)$ is lost, which makes $B$ and $a$ to not exist, which makes the SDE non integrable. So, using an Ito stochastic process as candidate solution for ill-posed equations, we can satisfy the constraint as close as we want, but not exactly. In other words, for every $\epsilon >0$ there exists a stochastic process, $du = a(x,u) dt + B(x,u) dW_t$ with $a$ and $B$ given by equation (\ref{equation:gen_ito_solution}), such that $||g(x(t))||<\epsilon$.
\end{remark}
Based on this discussion, the steps for computation are summarized in table \ref{alg:gen-ito}. The first step in table \ref{alg:gen-ito}, requires one to choose the function $y:\mathbb{R}^+\times\mathbb{R}^m\to \mathbb{R}^n$ such that $||g(y(0,u))||<\epsilon$. However, obtaining such a function is a difficult task and relies on trial and error. Alternatively, since the function $g$ is Lipschitz continuous, it may be possible to control the bound of $||g(y(0,u))||$ by controlling the bound on $||y(0,u)||$. Then a heuristic algorithm can be designed to ensure $||g(y(0,u))||<\epsilon$. We do not explore this in this thesis.
\begin{table}[h!]
\hrule
\vspace{0.5em}
\caption{Steps for converting completely high index SDAE to an SDE, whose solution satisfies $||g(x(t))||<\epsilon$.}
\label{alg:gen-ito}
\hrule
\begin{algorithmic}[1]
\STATE{Given $\epsilon>0$, choose $y:\mathbb{R}^+\times\mathbb{R}^m\to \mathbb{R}^n$ such that $||g(y(0,u))||<\epsilon$ and $( Dg D_vy)$ is invertible.}
\STATE{$B = (( Dg) D_vy)^{-1}(( Dg)\sigma)$}
\STATE{$\Lambda = \sigma -  D_vy B$}
\STATE{$\chi_i =(\partial_{x_w}g_i)(f_w -0.5B_{km}(\partial^2_{v_jv_k}y_w)B_{jm}-(\partial_{v_l}\Lambda_{wj})B_{lj})B_{lj})+ 0.5 \Lambda_{km}(\partial^2_{x_jx_k}g_i)\Lambda_{jm}$}
\STATE{$a = (( Dg) D_vy)^{-1}\chi$}
\STATE{Use Netwon's iteration to solve for $u_0$ in $g(y(0,u_0)) = 0$.}
\STATE{$d(x,u) = (f,a)dt + (\sigma,B)dW_t$ is the required SDE with the initial condition $(x(0),u_0)$.}
\end{algorithmic}
\hrule
\end{table}
\begin{remark}
Appropriate choice of the function $y(0,v)$ allows us to decide the interval of time upto which the solution is numerically trackable.
\end{remark}
\subsubsection{Example}
It is clear from the previous discussion that in case of completely high index equation, this method gives an SDE whose solution is the approximate solution for the SDAE. However, depending on the original SDE of the SDAE and the choice of the function $y(t,v)$, the SDE obtained from the method may or may not satisfy the growth bound condition for the SDE (refer to \cite{oksendal2003stochastic} for details on growth bound condition of an SDE). In this section we illustrate the method by the following SDAE.
\begin{subequations}
\begin{equation}
\label{equation:example-nonauto-SDE-UNSDAE}
dx = (x+x^2 + u) dt + 0.2 dB_t,
\end{equation}
\begin{equation}
\label{equation:example-nonauto-algebraic-eq-UNSDAE}
g(x,t) = 2x - x^3 - 0.5sin(4t) = 0,
\end{equation}
\end{subequations}
where, $g(x,t)$ is the algebraic equation, $u$ is the algebraic variable, $B_t$ is a 1-dimensional Wiener process, and $x(0) = 0$ is the initial condition for the SDE. Figure \ref{fig:UNSDAE-wto-u} shows $x(t)$ vs. $t$ plot for the unconstrained SDE (\ref{equation:example-nonauto-SDE-UNSDAE}) and the constraint equation.
\begin{figure}[h]
\includegraphics[width=\linewidth]{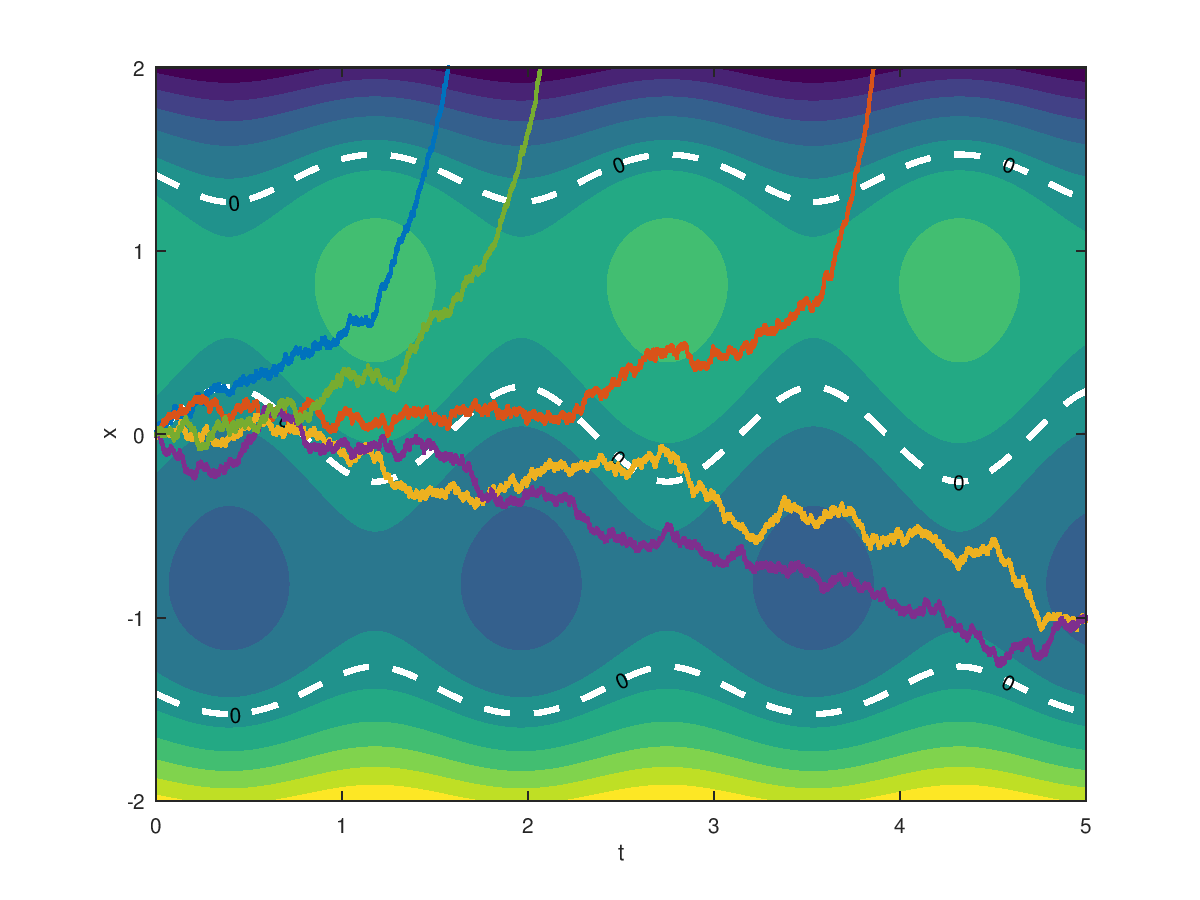}
\caption{$x(t)$ vs. $t$ plot for the unconstrained SDE (\ref{equation:example-nonauto-SDE-UNSDAE}). The white dashed line indicates the algebraic equation (\ref{equation:example-nonauto-algebraic-eq-UNSDAE}) and the colored lines show $x(t)$ vs. $t$ plot for different realizations. The numerical integration was performed using Euler-Maruyama method with time stepping of $10^{-4}$.}
\label{fig:UNSDAE-wto-u}
\end{figure}
The SDAE is a non autonomous completely high index SDAE. This SDAE can be converted into the following equivalent autonomous SDAE.
\begin{subequations}
\label{equation:auto-UNSDAE-example}
\begin{equation}
d \begin{bmatrix} x_1\\ x_2 \end{bmatrix} = \begin{bmatrix} x_1 + x_{1}^{2} + u\\ 1 \end{bmatrix} dt + \begin{bmatrix} 0.2 \quad 0\\ 0\quad 0 \end{bmatrix} dW_t,
\end{equation}
\begin{equation}
g\left(\begin{bmatrix}
x_1 \\ x_2
\end{bmatrix}\right) = 2x_1 -x_1^3 - 0.5sin(4x_2) = 0,
\end{equation}
\end{subequations}
where $W_t$ is a 2-dimensional Wiener process, and $(x_1(0),x_2(0)) = (0,0)$ is the initial condition for the SDE.

This gives an autonomous completely high index SDAE, in which the noise does not admit algebraic variable. So, by definition, the SDAE is an UNSDAE. Now, $Dg(x_1,x_2) = \begin{bmatrix}2 - 3x_1^2\quad -2cos(4x_2) \end{bmatrix}$ and $Img\left(\begin{bmatrix} 0.2 \quad 0\\ 0\quad 0 \end{bmatrix}\right) = \mathbb{R}\times\{0\}$.

Clearly, $\mathbb{R}\times\{0\}$ is not in $ker( Dg(x_1,x_2))$. Hence, from corollary \ref{corollary:no_sol_to_UNSAE}, there is no Ito stochastic solution for the given UNSDAE.

The first step in table \ref{alg:gen-ito} is to choose the function $y(0,u)$. If $y(0,u) = \begin{bmatrix}
-\dfrac{\epsilon}{4\pi}\tan^{-1}u \\ 0
\end{bmatrix}$ then for any $\epsilon<1$, $||g\left(\begin{bmatrix}
-\dfrac{\epsilon}{4\pi}\tan^{-1}u \\ 0
\end{bmatrix}\right)||<\epsilon$. Therefore, let $\begin{bmatrix} y_1\\ y_2 \end{bmatrix} = \begin{bmatrix} -\dfrac{\epsilon}{4\pi}\tan^{-1}u \\ 0 \end{bmatrix}$. With this, according to the steps in table \ref{alg:gen-ito},
$B = \left[ 0.2(\dfrac{\partial y_1}{\partial u})^{-1}\quad 0\right] = \left[\dfrac{-0.8\pi (1+u^2)}{\epsilon}\quad 0\right]$, and
\[a = \left(\dfrac{\partial y_1}{\partial u}\right)^{-1}\left(f -0.5B^2\dfrac{\partial^2 y_1}{\partial u^2}\right) + \left(\dfrac{\partial g}{\partial x_1}\dfrac{\partial y_1}{\partial u}\right)^{-1}\dfrac{\partial g}{\partial x_2}\]
Therefore,
\begin{equation}
\label{equation:UNSDAEtoSDE-gen-ito-U}
d \begin{bmatrix} x_1\\ x_2 \\ u\end{bmatrix} = \begin{bmatrix} x_1 + x_{1}^{2} + u\\ 1\\a\end{bmatrix} dt +\begin{bmatrix} 0.2 \quad\quad\quad\quad 0\\ 0\quad\quad\quad\quad 0 \\ -\dfrac{0.8\pi}{\epsilon}(1+u^2) \quad 0\end{bmatrix} dW_t
\end{equation}
is the required SDE with initial condition $(x_1(0),x_2(0),u(0)) = (0,0,0)$.
Now the resulting SDE (\ref{equation:UNSDAEtoSDE-gen-ito-U}) can be easily solved with the given initial condition.  Figure \ref{fig:UNSDAE-genito} shows $x_1$ vs. $x_2$ plot for the SDE (\ref{equation:UNSDAEtoSDE-gen-ito-U}). The numerical integration is performed using Euler-Maruyama method with time stepping of $10^{-6}$. However, this may only give a local solution as the growth bound condition is not checked.

\begin{figure}[h]
\includegraphics[width=\linewidth]{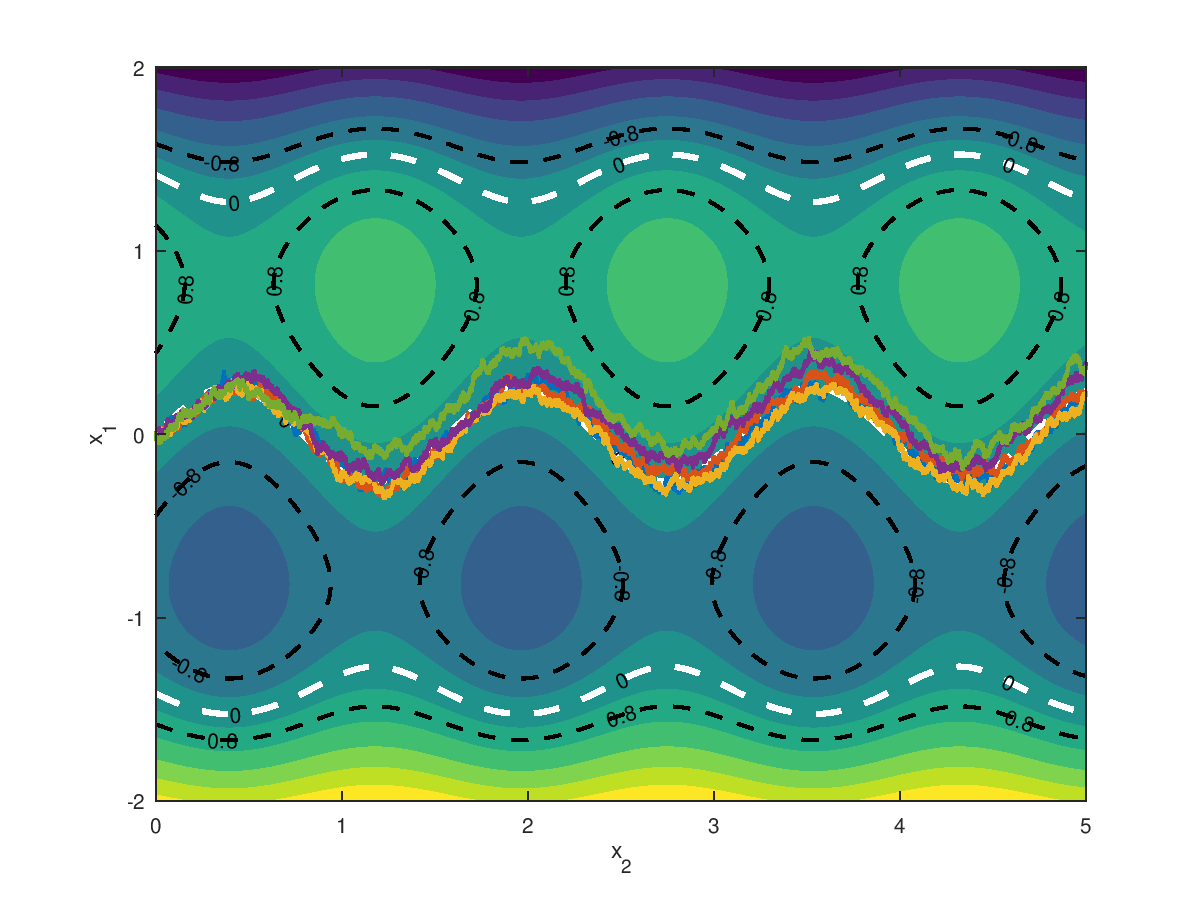}
\caption{$x_1$ vs. $x_2$ plot for the constrained SDE (\ref{equation:UNSDAEtoSDE-gen-ito-U}). The white dashed line indicates the algebraic constraint, black dashed lines show the bound on the constraint, and the colored lines show $x_1$ vs. $x_2$ plot for different realizations.}
\label{fig:UNSDAE-genito}
\end{figure}

\subsection{Approximate solution with bounded probability}
\label{sec:bounded_ito}
The method proposed in section \ref{sec:gen_ito} gives an approximate solution such that the constraint is satisfied with probability 1. In this section we use the ideas of stable control of stochastic differential equation and apply it to the SDAE setting. From the control theory view point, these ideas have already been explored earlier; e.g. in \cite{1967stochastic}, wherein the SDE is constrained to lie within a subset upto a probability bound. The key idea in \cite{1967stochastic} is to choose a Lyapunov function to obtain stable control solution so that SDE is constrained to lie on the constraint set. In this section, we use this idea in context of SDAEs by considering the set to be an $\epsilon$-neighbourhood of the constraint function. We present steps to find an approximate solution such that the probability that the approximate solution violates the constraint bound is bounded. For simplicity we will consider equation (\ref{eq:SDAE2})
\begin{subequations}
\begin{equation}
dx = f(x,u)dt + \sigma(x,u)dW_t,
\tag{\ref{eq:SDAE2-SDE}}
\end{equation}
\begin{equation}
0 = g(x) + \int_0^t\Gamma(x(s))dW_s;
\tag{\ref{eq:SDAE2-AE}}
\end{equation}
\end{subequations}
with $U = \mathbb{R}^n$ and $V = \mathbb{R}^m$.
\begin{definition}
Let $\lambda(t) = g(x(t)) + \int_0^t\Gamma(x(s))dW_s$ in equation (\ref{eq:SDAE2}). We say that a completely high index SDAE \eqref{eq:SDAE2} has a mean solution, or simply \textbf{{m-solution}}, if there exists $(x(t),u(t)) \in \mathbb{R}^n\times \mathbb{R}^m$ such that $\mathbb{E}(\lambda(t)) = 0$.
\end{definition}
\noindent Ito differentiation of the algebraic equation gives, $d\lambda = h(x,u)dt + ( D_xg(x)\sigma(x,u) + \Gamma(x))dW_t$, where $h(x,u)= ( D_x(g)(x))f(x,u) + \dfrac{1}{2}Tr(\sigma^T(x,u) ( D^2_{x,x}(g)(x))\sigma(x,u))$. If $h(x_t,u_t) = 0$ and $\lambda(0) = 0$, then the SDAE has a local m-solution. This is summarized in the following proposition.
\begin{proposition}
Consider completely high index SDAE (\ref{eq:SDAE2}) such that $f(x,u)$, $\sigma(x,u)$ and $\Gamma(x)$ are $\mathcal{C}^1$, $g(x)$ is $\mathcal{C}^3$. For some initial condition $u_0\in \mathbb{R}^m$ and $x_0\in \mathbb{R}^n$, assume that $h(x_0,u_0) = 0$ and $ D_uh(x_0,u_0)$ is an isomorphism, where $h:\mathbb{R}^n\times \mathbb{R}^m \to \mathbb{R}^m$ such that
\[h(x,u)= ( D_x(g)(x))f(x,u) + 0.5Tr(\sigma^T(x,u) ( D^2_{x,x}(g)(x))\sigma(x,u)).\]
Then there exists a local m-solution for the SDAE.
\end{proposition}
\begin{proof}
We know that \[\mathbb{E}(\lambda(t)) = \mathbb{E}\int_0^t h(x(x),u(x)) ds.\]
If $h(x_0,u_0) = 0$ and $D_uh(x_0,u_0)$ is an isomorphism, then implicit function theorem completes the proof.
\end{proof}
\noindent Even if an m-solution can be obtained, since the constraint is satisfied only in the mean sense, we do not have control over noise. Ideally, we would like to have a solution that is not just m-solution but also reduces the variance.
\begin{definition}
Let $\lambda(t) = g(x(t)) + \int_0^t\Gamma(x(s))dW_s$ for completely high index SDAE (\ref{eq:SDAE2}). By $\epsilon$-{\textbf{bounded m-solution}} with probability $\alpha$, we mean that there exists an m-solution $(x(t),u(t))\in \mathbb{R}^n\times \mathbb{R}^m$ for $t\in (0,a]$, such that $\mathbb{P}(\sup_{s\in (0,a]}||\lambda(s)||>\epsilon)\leq \alpha$ for some $\alpha\in (0,1]$ and $\epsilon >0$.
\end{definition}
\begin{proposition}
\label{proposition:bounded-m-solution}
Consider completely high index SDAE (\ref{eq:SDAE2}) such that $f(x,u)$, $\sigma(x,u)$ and $\Gamma(x)$ are $\mathcal{C}^1$, $g(x)$ is $\mathcal{C}^3$. Let $h:\mathbb{R}^n\times\mathbb{R}^m \to \mathbb{R}^m$ such that \[h(x,u)= ( D_x(g)(x))f(x,u) + \dfrac{1}{2}Tr(\sigma^T(x,u) ( D^2_{x,x}(g)(x))\sigma(x,u)).\]
Suppose there exists semi-martingale $(x_t,u_t)\in \mathbb{R}^n\times\mathbb{R}^m$ that it not only satisfies the SDE (\ref{eq:SDAE2-SDE}) but also satisfies \[h(x_t,u_t) = -b\lambda(t)\] in some interval of existence $(0,a]$, where $b>0$, \[\lambda(t) = g(x(t)) + \int_0^t\Gamma(x(s))dW_s,\] \[U\times V = \{(x_t(\omega),u_t(\omega))\in \mathbb{R}^n\times\mathbb{R}^m| t\in (0,a], \omega\in \Omega\}\] with $\Omega$ as the sample space, and $\sup_{(x,u)\in U\times V} ||( D_x(g)(x))\sigma(x,u) +\Gamma(x,u))||^2  <\infty$. Then $(x_t,u_t)$ is a local m-solution for equation (\ref{eq:SDAE2}). Moreover, for every $\epsilon>0$ and $\alpha\in (0,1]$, $\exists$ $b>0$ such that  ${\mathtt{P}(\sup_{s\in (0,t]}||\lambda(s)||>\epsilon)\leq \alpha}$ for all $t\in (0,a]$.
\end{proposition}
\begin{proof}
The proof is based on the conventional proof for stability of SDEs.
This idea has been discussed in \cite{1967stochastic}. Our solution can be interpreted as a special case of the Lyapunov stable solution from \cite{1967stochastic}, obtained using a quadratic Lyapunov function.

Following the proof of existence of m-solution, as $h(x,u)$ is $\mathcal{C}^1$ and $ D_uh(x_0,u_0)$ is an isomorphism, by implicit function theorem locally there exists $k(x,w)$ such that $h(x,k(x,w)) = w$. With $u_t = k(x(t),-b\lambda(t))$ we get,
\[d\lambda = -b\lambda dt + ( D_xg(x)\sigma(x,u) + \Gamma(x))dW_t.\]
As $\lambda(0) = 0,$
\[\lambda(t) =  e^{-bt}\int_0^te^{bs}( Dg\sigma + \Gamma) dW_s.\]
We know that $\lambda(t)e^{bt}$ is a martingale. Hence, $\mathtt{E}(\lambda(t)) = 0$ means that the solution is an m-solution. Furthermore, using Doob's martingale inequality, \[\mathtt{P}(\sup_{s\in (0,t]}||\lambda(s) e^{bs}||>k) \leq \dfrac{\mathtt{E}(\lambda^T(t)\lambda (t) e^{2bt})}{k^2}.\]
Using Ito-isometry property,
\[\mathtt{E}(\lambda^T\lambda e^{2bt}) = \mathtt{E}\int_0^te^{2bs} Tr(AA^T) ds,\]
where $A =  Dg\sigma + \Gamma$. Therefore,
\[\mathtt{P}(\sup_{s\in (0,t]}||\lambda(s) e^{bs}||>k)\leq \dfrac{\mathtt{E}\int_0^te^{2bs} Tr(AA^T) ds}{k^2}\leq \dfrac{J(e^{2bt} -1)}{2bk^2},\] where $J = \sup_{x\in U}\{Tr(AA^T)\}<\infty$, and $k>0$. If $k = e^{bt}\epsilon$, then 
\[\mathtt{P}(\sup_{s\in (0,t]}||\lambda(s) e^{bs}||>k) = \mathtt{P}\left(\sup_{s\in (0,t]}||\lambda(s)||>\epsilon\right)\]\[ \leq\dfrac{J(1-e^{-2bt})}{2b\epsilon^2} \leq \dfrac{J}{2b\epsilon^2}.\]
Therefore, if ${b = \dfrac{J}{2\epsilon^2\alpha}}$, then ${\mathtt{P}(\sup_{s\in (0,t]}||\lambda(s)||>\epsilon)\leq \alpha}$ i.e. given ${\alpha\in (0,1]}$ and ${\epsilon>0 \: \exists\: b>0}$ such that ${\mathtt{P}(\sup_{s\in (0,t]}||\lambda(s)||>\epsilon)\leq \alpha}$, with $d\lambda = -b\lambda dt + ( D_xg(x)\sigma(x,u) + \Gamma(x))dW_t$ and $\lambda(0) = 0$.
\end{proof}
\begin{remark}
From the engineering view point of the control theory, the bounded m-solution results in a proportional type control solution in the mean sense. In this context the value of $b$ can be interpreted as the proportional gain. As $\alpha \to 0$, $b\to \infty$. This indicates that it is not possible to eliminate noise by controlling the drift term. Similarly, if $\epsilon \to 0$ then $b \to \infty$. This confirms the fact that in case of ill-posed equations, there is no path-wise solution.
\end{remark}
The proof of proposition \ref{proposition:bounded-m-solution} relies on implicit function theorem to ensure that $h(x_t,u_t) = -b\lambda(t)$. Hence, numerical computation of bounded m-solution will involve the Newton's method of root finding. However, performing newton's iteration at each time step can be computationally inefficient. If a completely high index SDAE can be reduced to an SDAE of index 1 such that the solution of the two SDAEs agree approximately, then the method of lemma \ref{lem:SDEforSDAE} can be used (provided required conditions of the lemma are satisfied). In table \ref{alg:bounded-sol} we combine this idea with proposition \ref{proposition:bounded-m-solution}, and gives the steps to convert a given UNSADE to an index 1 SDAE.
\vspace{0.5em}
\begin{table}[h!]
\hrule
\vspace{0.5em}
\caption{Steps for converting the given high index SDAE into an index 1 SDAE to compute bounded m-solution.}
\label{alg:bounded-sol}
\vspace{0.5em}
\hrule
\begin{algorithmic}[1]
\STATE{$A(x) = ( Dg(x))\sigma(x)$}
\STATE{$J = \sup_{x\in U}\{Tr(A(x)A^T(x))\}$}
\STATE{Choose $b\in\mathbb{R}$ such that ${b > \dfrac{J}{2\epsilon^2\alpha}}$}.
\STATE{$h(x,u) = ( Dg(x))f(x,u) + \dfrac{Tr}{2}(\sigma^T(x)( D^2g(x))\sigma(x)) + b(g(x))$}
\STATE{If the given SDE with $h(x,u) = 0$ as the constraint make an index 1 SDAE, then it can be solved trivially.}
\end{algorithmic}
\hrule
\end{table}

\subsubsection{Examples}
\textbf{\textit{Example 1}}:\\
Consider SDAE (\ref{equation:auto-UNSDAE-example}) form the previous example. The goal is to construct an SDE that has a bounded m-solution such that $\mathbb{P} ( ||g(x(t))||>\epsilon )\leq \alpha$. Suppose $\alpha = 0.8$ and $\epsilon = 0.5$. In step 2 of algorithm {\ref{alg:bounded-sol}} $J = \sup_{x\in U}\{ Tr(A(x)A^T(x))\}$, where U is the neighbourhood of the initial condition. In the current example, let $U = (-2,2)\times(-5,5)$. Hence, $J = \sup_{x\in U}\{ 0.04(2-3x_1^2 )^2 \} = 4$, and $b>10$. Let $b = 11$.\\ The new index 1 SDAE that gives a bounded m-solution for the given SDAE is
\begin{subequations}
\begin{equation}
d \begin{bmatrix} x_1\\ x_2 \end{bmatrix} = \begin{bmatrix} x_1 + x_{1}^{2} + u\\ 1 \end{bmatrix} dt + \begin{bmatrix} 0.2 \quad 0\\ 0\quad 0 \end{bmatrix} dW_t,
\end{equation}
\begin{multline}
h(x,u) = (\mathcal{D}g(x))f(x,u)+ \dfrac{Tr}{2}(\sigma^T(x)(\mathcal{D}^2g(x))\sigma(x)) + b(g(x))\\
= -2cos(4x_2) + (2-3x_1^2)(x_1 + x_1^2 + u) - 0.12x_1 + 11g(x_1,x_2)=0
\end{multline}
\end{subequations}
Now either lemma (\ref{lem:SDEforSDAE}) can be used to arrive at an SDE for the SDAE or Newton's iteration can be used at every time step to solve the constraint for $u$. In the current example, $u$ can be trivially obtained from the algebraic equation $h(x,u) = 0$.
\begin{multline}
u = \dfrac{2cos(4x_2) - 11g(x)}{(2 - 3x_1^2)} +  \dfrac{0.12x_1 - (2 - 3x_1^2)(x_1 + x_1^2)}{(2 - 3x_1^2)},
\end{multline}
as long as $x_1\neq \pm\sqrt{2/3}$.
Figure \ref{fig:UNSDAE-bounded-m-sol} shows the $x_1$ v/s $x_2$ plot for a bounded m-solution of SDAE \ref{equation:auto-UNSDAE-example}. The numerical integration is performed using Euler-Maruyama method with time stepping of $10^{-4}$.
\begin{figure}[h]
\includegraphics[width=\linewidth]{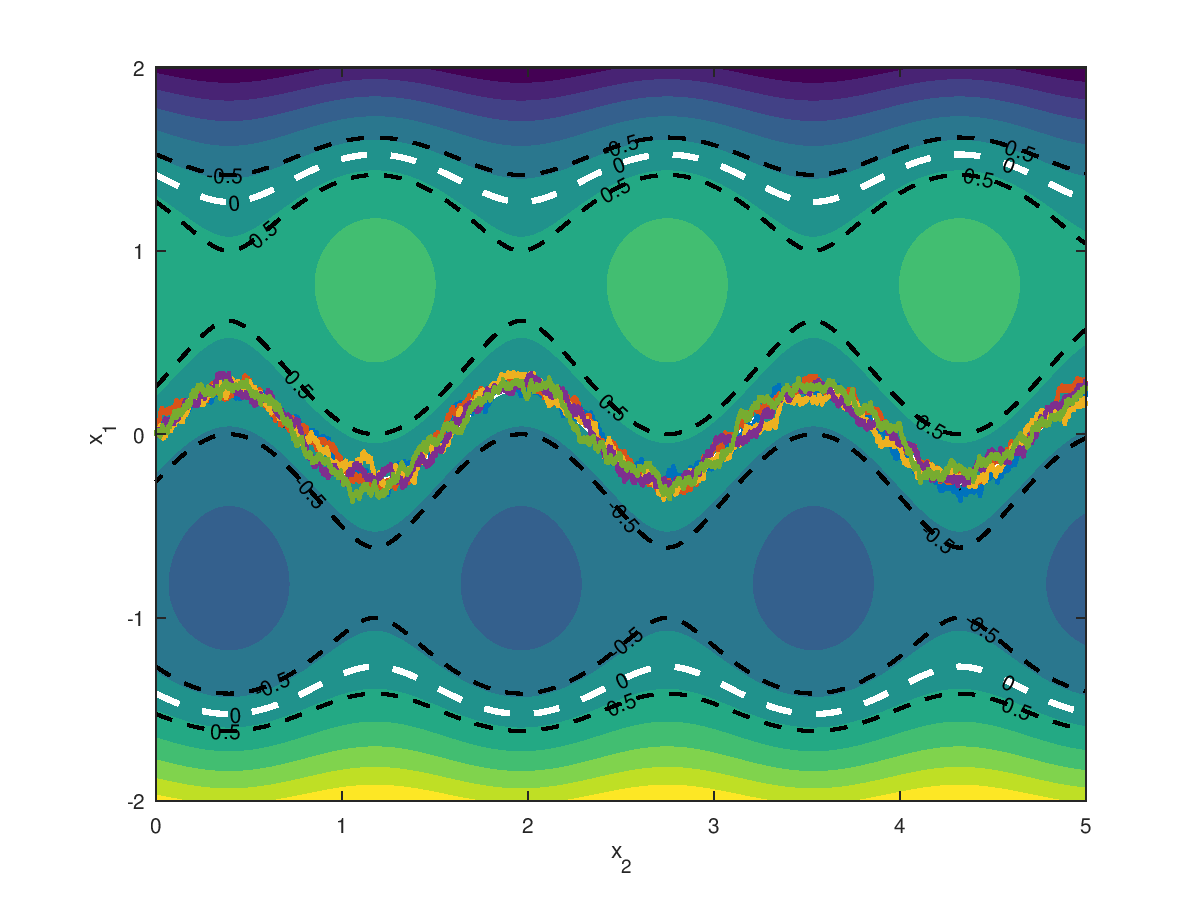}
\caption{$x_1$ v/s $x_2$ plot for a bounded m-solution of SDAE \ref{equation:auto-UNSDAE-example}. The white dashed line indicates the algebraic equation and the colored lines show $x_1$ v/s $x_2$ plot for different realizations.}
\label{fig:UNSDAE-bounded-m-sol}
\end{figure}
\medskip
\\
\noindent \textbf{\textit{Example 2}}:\\
We will consider another example of controlling the trajectory of a particle in gravitational field due to a single body in 2-dimensional space. Suppose the object of unit radius is located the origin and produces a gravitational acceleration field, which is given by \[G(x) = -\dfrac{5x}{||x||^3}.\]
Hence, the in presence of additive noise, the dynamics is given by
\begin{subequations}
\label{equation: example SDE_2}
\begin{equation}
\dot{x} = v
\end{equation}
\begin{equation}
d{v} = \left(G(x) + u\right)dt + \sum_{l = 1}^2\sigma dW^l_t,
\end{equation}
\end{subequations}
where $\sigma = 0.3$. Physically, this noise can be due to noisy position feedback data or due to unaccounted forces like photonic pressure, gravitational modelling errors etc. Our objective is to make the particle follow the trajectory that is given by \[\phi(t) = \begin{bmatrix} -3cos(t) +4cos(3t/4)\\ -3sin(t) +4sin(3t/4) \end{bmatrix}.\]
Hence, the constraint is 
\begin{equation}
g(x,t) = x-\phi(t).
\end{equation}
In figure \ref{fig:SDE-ex2-without_control}, we show the plot for the constraint and the SDE \eqref{equation: example SDE_2} with $u = 0$.
\begin{figure}[h]
\begin{center}
\includegraphics[width=\linewidth]{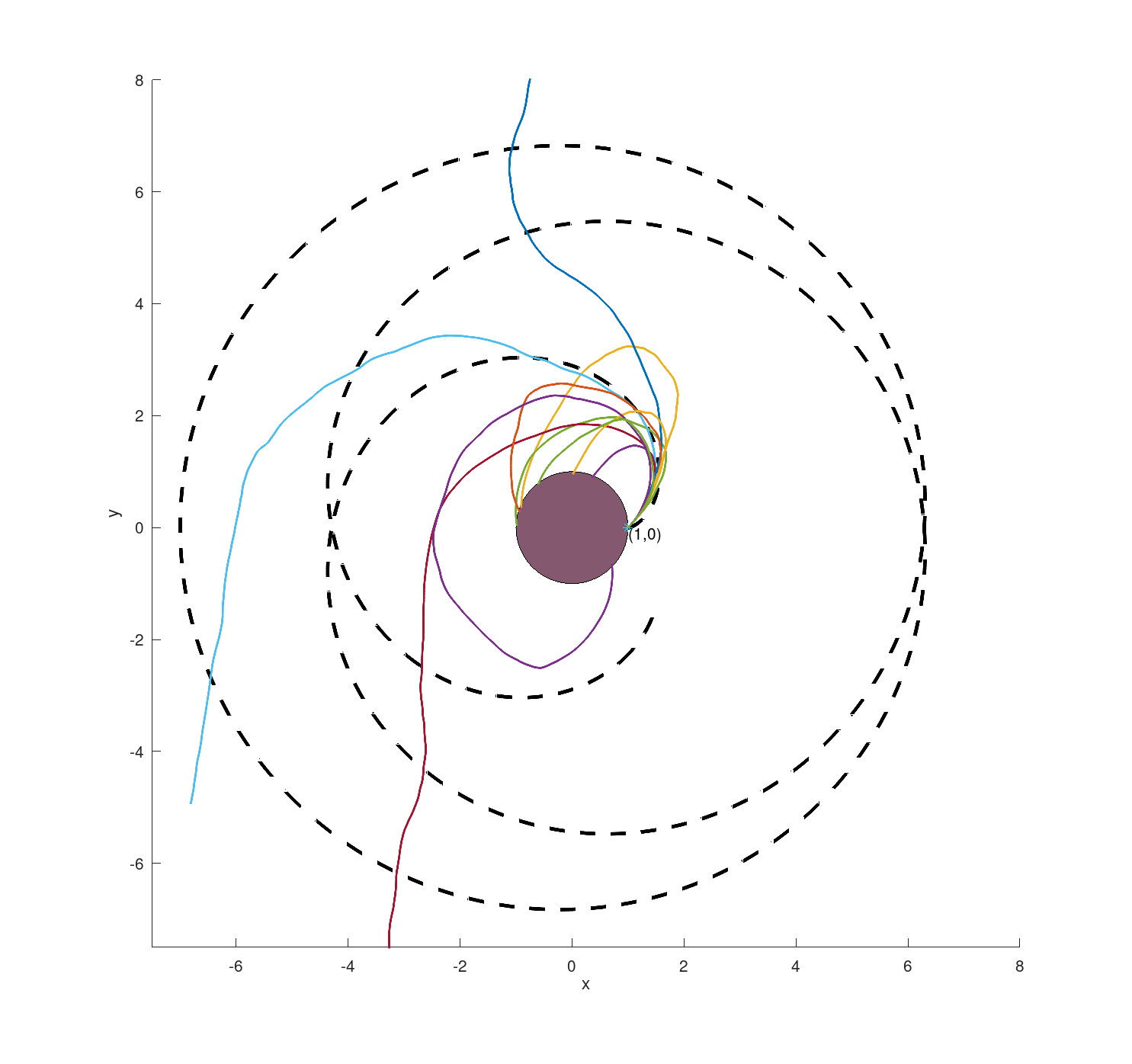}
\end{center}
\vspace{-2em}
\caption{Trajectory plot $(x(t))$ for a particle following equation \eqref{equation: example SDE_2} with $u = 0$. The initial condition for the SDE is taken as $(x(0),v(0)) = ((1,0),(1.8,1.8))$. The black dashed line indicates the constraint trajectory $\phi(t)$ and the colored lines show the trajectory plot for $x(t)$ in different realizations. The unit circle at the origin indicates the object responsible for the gravitational field.}
\label{fig:SDE-ex2-without_control}
\end{figure}
The Ito differentiation of the constraint gives,
\[g(x(t),t) = g(x(0),0) + \int_0^t(v(s) - \phi'(s))ds.\]
Therefore, if $g(x(0),0) = 0$ and $v(s) = \phi'(s)$ for all $s\in [0,t]$ then $g(x(t),t) = 0$. Moreover, when $g(x(0),0) = 0$, as $\sup_{z\in[0,t]}||g(x(z),z)|| = \sup_{z\in[0,t]}||\int_0^z(v(s) - \phi'(s))ds||\leq t\sup_{s\in[0,t]}||(v(s) - \phi'(s))||$,
\begin{equation}
\mathbb{P}(\sup_{z\in[0,t]}||g(x(z),z)||>\epsilon)\leq \mathtt{P}(t\sup_{s\in[0,t]}||(v(s) - \phi'(s))||>\epsilon) = \mathtt{P}(\sup_{s\in[0,t]}||(v(s) - \phi'(s))||>\epsilon /t).
\end{equation}
With $h(x,v,t) = v - \phi'(t)$ as the new constraint, we need to find the bounded m-solution, upto $\epsilon /t$ with probability $\alpha$.
\begin{figure}[h]
\begin{center}
\includegraphics[width=0.9\linewidth]{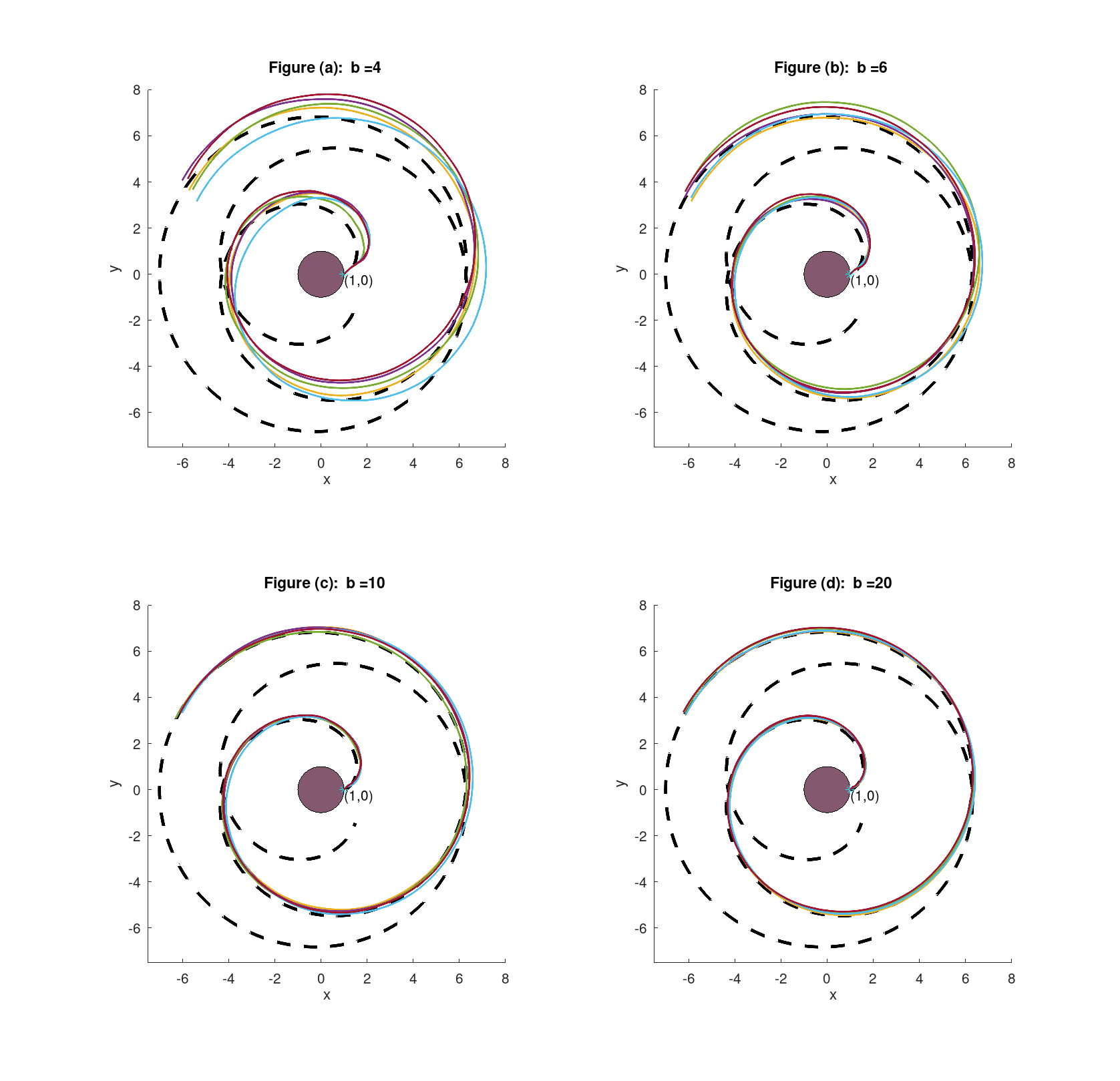}
\end{center}
\vspace{-2em}
\caption{Trajectory plot $(x(t))$ when obtained as a bounded m-solution of the SDAE \ref{equation: example2 SDAE}. The initial condition for the SDE is taken as $(x(0),v(0)) = ((1,0),(0,0))$. The black dashed line indicates the constraint trajectory $\phi(t)$ and the colored lines show the trajectory plot for $x(t)$ in different realizations. The unit circle at the origin indicates the object responsible for the gravitational field. Each plot is for different value of $b$.}
\label{fig:SDAE-ex2}
\end{figure}
From algorithm \ref{alg:bounded-sol}, we know that if we choose $b>\dfrac{Jt^2}{2\epsilon^2 \alpha},$ then we can ensure that $\mathbb{P}(\sup_{z\in[0,t]}||g(x(z),z)||>\epsilon)\leq \alpha.$ We will choose large enough $b$ to avoid explicit computation of $J$. Figure \ref{fig:SDAE-ex2} shows trajectory plots for different realizations for different values of $b$ for SDAE \ref{equation: example2 SDAE}.
\begin{subequations}
\label{equation: example2 SDAE}
\begin{equation}
\dot{x} = v,
\end{equation}
\begin{equation}
d{v} = \left(G(x) + u\right)dt + \sum_{l = 1}^2\sigma dW^l_t,
\end{equation}
\begin{equation}
h(x,v,t) = v-\phi'(t),
\end{equation}
\end{subequations}
where $\sigma = 0.3$ and $\phi(t) = \begin{bmatrix} -3cos(t) +4cos(3t/4)\\ -3sin(t) +4sin(3t/4) \end{bmatrix}$.\\
The solution trajectory of the constraint problem is shown in figure \ref{fig:SDAE-ex2}. As observed in figure \ref{fig:SDAE-ex2}, as the value of $b$ increases, the approximation in constraint satisfaction is reduced; i.e., as $b\to \infty$, \[\mathbb{P}(\sup_{z\in[0,t]}||g(x(z),z)||>\epsilon)\to 0\] for any value of $\epsilon>0$ and $\alpha\in (0,1]$.
\subsection{Comparison of the approximate methods}
\label{sec:comparison of techni.}
The advantage of the first method given in section \ref{sec:gen_ito} is that the constraint bound is satisfied with probability 1. The first method works by converting the given SDAE into a stochastic partial differential equation using the method of characteristic. Then the function for the state vector is chosen such that the stochastic partial differential equation for the constraint has a bounded solution. The algebraic variable is found by freezing the dynamics of the constraint stochastic partial differential equation. When the stochastic partial differential equation is converted back to the SDE, it is expected that the constraint bound is still respected. However, this fails to happen if there is a finite time explosion of the resulting SDE for the algebraic equation. Upto a certain extent, the finite time explosion can be controlled by an appropriate choice of the function $y(0,u)$. On the other hand, the second technique given in section \ref{sec:bounded_ito} gives a minimum probability (but not unit probability) with which the constraint bound is satisfied. The key idea, of Lyapunov stability of SDEs, used in this method has been well documented in literature to obtain a solution for SDEs constrained to sets. We have used this idea for SDAEs. Compared to method 1, numerical implementation of this method seems simpler due of fewer terms than that in method 1. Choice between the two techniques should be made on the basis of the criticality of the application.

\section{Concluding remarks}
\label{sec:conclusions}
In this article we have attempted to generalize the notion of index from the theory of Differential-Algebraic Equations to the stochastic setting. We have given the sufficient condition for existence of a unique solution for Stochastic Differential Algebraic Equations for SDAEs with same number of algebraic equations as that of the algebraic variables. We have also given necessary conditions for the existence of the solution. Based of the necessary condition, we establish the existence of ``ill-posed" equations for which there is no solution. Since it is very common to encounter ill-posed SDAEs in many applications, techniques to find approximate solution for all completely high index SDAEs are necessary. Therefore, we proposed two different methods that can be used to solve (approximately) any type of completely high index equation (including deterministic DAEs, in which the noise terms can be simply considered zero). With this we conclude our work.

\section*{Acknowledgment}
This work was partially supported by the grant SERB-EMR/2016/002022-EEC of Science and Engineering Research Board, Government of India.

\bibliography{references}

\end{document}